\documentclass[12pt, 
]{amsart}

\usepackage[all]{xy}
\usepackage{mathrsfs, amsthm, amscd, amsmath, amssymb, graphicx}

\usepackage[dvipdfmx]{hyperref}

\newtheorem{theo}{Theorem}[section]
\newtheorem{prop}[theo]{Proposition}
\newtheorem{lemm}[theo]{Lemma}
\newtheorem{cor}[theo]{Corollary}
\newtheorem{claim}[theo]{Claim}
\newtheorem{prob}[theo]{Problem}

\theoremstyle{definition}
\newtheorem{defi}[theo]{Definition}

\newtheorem{step}{Step}
\theoremstyle{remark}
\newtheorem{rem}[theo]{Remark}

\makeatletter
    
    \@addtoreset{equation}{section}
\makeatother

\newcommand{\Image}[0]{\operatorname{Im}}

\newcommand{\deldel}{\sqrt{-1}\partial \overline{\partial}}
\newcommand{\dbar}{\overline{\partial}}
\newcommand{\e}{\varepsilon}
\newcommand{\ome}{\widetilde{\omega}}
\newcommand{\I}[1]{\mathcal{I}(#1)}
\newcommand{\nd}[1]{{\rm{nd}}(#1)}

\newcommand{\lla}[0]{{\langle\!\hspace{0.02cm} \!\langle}}
\newcommand{\rra}[0]{{\rangle\!\hspace{0.02cm}\!\rangle}}
\newcommand{\la}[0]{{(\!\hspace{0.02cm} \!(}}
\newcommand{\ra}[0]{{)\!\hspace{0.02cm}\!)}}

\begin{document}

\title[Injectivity theorems for higher direct images]
{Injectivity theorems with multiplier ideal sheaves \\
for higher direct images under K\"ahler morphisms}

\author{Shin-ichi MATSUMURA}
\address{Mathematical Institute, Tohoku University, 
6-3, Aramaki Aza-Aoba, Aoba-ku, Sendai 980-8578, Japan.}
\email{{\tt mshinichi-math@tohoku.ac.jp, mshinichi0@gmail.com}}
\date{\today, version 0.01}
\renewcommand{\subjclassname}{%
\textup{2010} Mathematics Subject Classification}

\maketitle

\begin{abstract}
The purpose of this paper is to establish injectivity theorems 
for higher direct image sheaves 
of canonical bundles twisted by pseudo-effective line bundles 
and multiplier ideal sheaves. 
As applications, 
we generalize Koll\'ar's torsion freeness and  
Grauert-Riemenschneider's vanishing theorem. 
Moreover, we obtain a relative vanishing theorem 
of Kawamata-Viehweg-Nadel type 
and an extension theorem for holomorphic sections 
from fibers of morphisms to the ambient space. 
Our approach is based on transcendental methods    
and works for K\"ahler morphisms and 
singular hermitian metrics with non-algebraic singularities. 
\end{abstract}

\tableofcontents

\section{Introduction}\label{Sec1}
The injectivity theorem, 
which has been studied in the last decades,  
is a very powerful tool 
to study higher dimensional algebraic geometry 
(in particular birational geometry) and complex geometry. 
After the pioneering work by Tankeev in \cite{Tan71}, 
Koll\'ar established the celebrated injectivity theorem in \cite{Kol86a} 
by using the Hodge theory. 
From the viewpoint of Hodge theory, 
we have already obtained many useful generalizations   
(for example, see \cite{Amb03}, \cite{Amb14}, \cite{EV92}, \cite{Fuj09}, \cite{Fuj11b}, 
\cite{Fuj13b}, \cite{Fuj14a}, and \cite{Kol86b}). 
Particularly, 
the following theorem is one of the most useful generalizations of 
Koll\'ar's injectivity theorem for deformations of projective varieties. 
On the other hand, also from the analytic viewpoint, 
we can approach to Koll\'ar's result  
(for example, see \cite{Eno90}, \cite{Fuj12}, \cite{Fuj13},  \cite{FM16},  
\cite{Mat13}, \cite{Mat14}, \cite{Ohs04}, \cite{Tak95}, and \cite{Tak97}). 
This paper contributes to the study of 
the injectivity theorem and its applications 
from the analytic viewpoint. 

\begin{theo}\label{kol}
Let $\pi \colon X \to \Delta$ be a surjective projective morphism 
from a smooth variety $X$ 
to a quasi-projective variety  $\Delta$, 
and $F$ be a $\pi$-semi-ample line bundle on $X$.

Then, for a non-zero $($holomorphic$)$ section $s$ of $F^m$ $(m \geq 0)$,  
the multiplication map induced by the tensor product with $s$ 
$$
R^q \pi_{*}(K_{X}\otimes F) 
\xrightarrow{\otimes s} 
R^q \pi_{*}(K_{X}\otimes F^{m+1}) 
$$
is injective for every $q$. 
Here $K_{X}$ denotes the canonical bundle of $X$ 
and $R^q \pi_{*}(\bullet)$ denotes 
the $q$-th higher direct image sheaf. 
\end{theo}

In this paper, we consider a proper K\"ahler morphism $\pi \colon X \to \Delta$ 
from a complex manifold $X$ to an arbitrary analytic space $\Delta$ 
and a (holomorphic) line bundle $F$ on $X$ equipped with a singular (hermitian) metric $h$, and we study the direct image sheaves 
$R^q \pi_{*}(K_{X}\otimes F \otimes \I{h})$ of the canonical bundle $K_{X}$ on $X$ 
twisted by $F$ and the multiplier ideal sheaf $\I{h}$ of $h$. 
As results, we establish two injectivity theorems  
formulated for singular metrics with arbitrary singularities, 
which can be seen as a generalization of Theorem \ref{kol} 
to pseudo-effective line bundles 
(see Theorem \ref{main} and Theorem \ref{main2}).
As applications, we give a generalization of 
Koll\'ar's torsion freeness and 
Grauert-Riemenschneider's vanishing theorem 
(see Corollary \ref{cor-1}). 
Moreover, we obtain a relative vanishing theorem 
of Kawamata-Viehweg-Nadel type (see Theorem \ref{KVN})
and an extension theorem for (holomorphic) sections (see Corollary \ref{cor-2}).

In \cite{Eno90}, Enoki obtained the special case (the absolute case) of Theorem \ref{kol} 
under the weaker assumption that $F$ is semi-positive  
(namely, it admits a smooth (hermitian) metric with semi-positive curvature), 
as an application of the theory of harmonic integrals. 
Takegoshi proved the relative case of Enoki's result in \cite{Tak95}, 
which is a complete generalization of Theorem \ref{kol} 
from semi-ample line bundles to semi-positive line bundles. 
In this paper, 
we handle line bundles admitting a (possibly) singular metric 
with semi-positive curvature (that is, pseudo-effective line bundles). 
The study of pseudo-effective line bundles is one of 
the important subjects, and thus it is natural and of interest 
to study further generalizations of Theorem \ref{kol} and Takegoshi's result 
from semi-positive line bundles to pseudo-effective line bundles.

The following theorems, which are the main results of this paper, 
can be seen as a generalization of Theorem \ref{kol} and Takegoshi's result 
to pseudo-effective line bundles. 
Moreover Theorem \ref{main} and Theorem \ref{main2} 
include various injectivity theorems, 
for example, \cite{Eno90}, \cite{FM16}, \cite{Fuj12}, 
\cite{Fuj13}, \cite{GM13}, \cite{Kol86a}, \cite{Mat13}, \cite{Mat14}, 
\cite{Tak95}, \cite{Tak97}, and so on.

\begin{theo}[Main Result A] \label{main} 
Let $\pi \colon X \to \Delta$ be a surjective 
proper K\"ahler morphism 
from a complex manifold $X$ to an analytic space $\Delta$, 
and $(F,h)$ be a $($possibly$)$ singular hermitian line bundle on $X$ 
with semi-positive curvature.

Then, for a non-zero $($holomorphic$)$ section $s$ of $F^m$ $(m \geq 0)$ 
satisfying $\sup_{K} |s|_{h^m} < \infty$
for every relatively compact set $K \Subset X$, 
the multiplication map 
induced by the tensor product with $s$ 
$$
R^q \pi_{*}(K_{X}\otimes F \otimes \I{h}) 
\xrightarrow{\otimes s} 
R^q \pi_{*}(K_{X}\otimes F^{m+1} \otimes \I{h^{m+1}}) 
$$
is injective for every $q$. 
Here $\I{\bullet}$ denotes the multiplier ideal sheaf of $\bullet$. 
\end{theo}

\begin{theo}[Main Result B] \label{main2} 
Let $\pi \colon X \to \Delta$ be a surjective 
proper K\"ahler morphism 
from a complex manifold $X$ to an analytic space $\Delta$. 
Let $(F,h)$ be a 
$($possibly$)$ singular hermitian line bundle on $X$ 
and $(M,h_{M})$ be a smooth hermitian line bundle on $X$. 
Assume that 
\begin{equation*}
\sqrt{-1}\Theta_{h_M}(M)\geq 0 \quad \text{and}\quad 
\sqrt{-1}(\Theta_h(F)-b \Theta 
_{h_M}(M))\geq 0
\end{equation*}
for some $b>0$.

Then, for a non-zero $($holomorphic$)$ section $s$ of $M$, 
the multiplication map induced by the tensor product with $s$ 
$$
R^q \pi_{*}(K_{X}\otimes F \otimes \I{h}) 
\xrightarrow{\otimes s} 
R^q \pi_{*}(K_{X}\otimes F \otimes \I{h} \otimes M) 
$$
is injective for every $q$. 
\end{theo}

\begin{rem}\label{rem-main2}
(1) For every point $t \in \Delta$ 
we can take an open neighborhood $\Delta'$ of $t$ 
such that $\pi^{-1}(\Delta')$ is a K\"ahler manifold, 
when $\pi \colon X \to \Delta$ is a K\"ahler morphism 
(for example see \cite[Definition 6.1]{Tak95} for the definition of K\"ahler morphisms). 
\vspace{0.1cm}
\\
(2) The case $m=0$ in Theorem \ref{main} agrees with 
the case where $(M,h_{M})$ is trivial in Theorem \ref{main2}. 
This case is important for applications. 
\vspace{0.1cm}
\\
(3) The assumption in Theorem \ref{main} on the local sup-norm  
is a reasonable condition to make the multiplication map well-defined 
and it is always satisfied in the case $m=0$.
\end{rem}

In \cite{Mat13} and \cite{FM16}, by combining the theory of harmonic integrals 
with the $L^2$-method for the $\dbar$-equation, 
we succeeded to obtain the above results in the absolute case 
(see also \cite{GM13} and \cite{FM16} for applications).
The proof of the main results is based on transcendental methods 
developed in \cite{Mat13}, \cite{FM16}, and \cite{Tak95}. 
One of the advantages of our method is that 
we can prove the main results for K\"ahler morphisms 
(not only projective morphisms) and 
singular metrics with non-algebraic singularities.  
We must sometimes handle the singular metric $h$ obtained from  
taking the limit of suitable metrics $\{h_{m}\}_{m=1}^{\infty}$. 
Under the regularity (smoothness) for singular metrics, 
a theorem similar to Theorem \ref{main2} was given in \cite{Fuj13}. 
However, it is quite hard 
to investigate the regularity of the limit $h$, 
even if $h_{m}$ has algebraic singularities. 
Therefore it is worth formulating Theorem \ref{main} and Theorem \ref{main2} for 
singular metrics with arbitrary singularities.

As a direct corollary, 
we generalize Koll\'ar's torsion freeness (\cite{Kol86a}) 
and Grauert-Riemenschneider's vanishing theorem (\cite{GR70})
for the higher direct images $R^q \pi_{*}(K_{X}\otimes F \otimes \I{h})$.

\begin{cor}[{\rm{Koll\'ar's torsion freeness, 
Grauert-Riemenschneider's vanishing theorem}}]\label{cor-1} 
\  
Let $\pi : X \to \Delta$ be a surjective 
proper  K\"ahler morphism 
from a complex manifold $X$ to an analytic space $\Delta$, 
and $(F,h)$ be a $($possibly$)$ singular hermitian line bundle on $X$ 
with semi-positive curvature.

Then the higher direct image sheaf 
$R^q \pi_{*}(K_{X}\otimes F \otimes \I{h})$ is torsion free 
for every $q$. 
Moreover, we obtain 
$$
R^q \pi_{*}(K_{X}\otimes F \otimes \I{h})=0
\quad \text{for every $q>\dim X - \dim \Delta$. }
$$
\end{cor}

As a further application, 
we obtain a vanishing theorem of Kawamata-Viehweg-Nadel type 
(\cite{Kaw82}, \cite{Vie82}, \cite{Nad90})
for the higher direct images 
$R^q \pi_{*}(K_{X}\otimes F \otimes \I{h})$ (see Theorem \ref{KVN}). 
For the proof of Theorem \ref{KVN}, 
we need the lower semi-continuity of 
the numerical Kodaira dimension of singular hermitian line bundles, 
which is of independent interest. 
If we can prove Proposition \ref{lsc} for K\"ahler morphisms,  
we will be able to generalize Theorem \ref{KVN} for them.  
See Definition \ref{nd} or \cite{Cao14} 
for the definition of the numerical Kodaira dimension of singular hermitian line bundles. 

\begin{prop}[Quasi lower semi-continuity of the numerical Kodaira dimension]\label{lsc}
\ \\
Let $\pi \colon X \to \Delta$ be 
a surjective projective morphism 
from a complex manifold $X$ to an analytic space $\Delta$, 
and $(F,h)$ be a $($possibly$)$ singular hermitian line bundle on $X$ 
with semi-positive curvature. 
Assume that $\pi$ is smooth at a point $t_{0} \in \Delta$.

Then, there exist an open neighborhood $B$ of $t_{0}$ and 
a dense subset $Q \subset B$ 
with the following property\,$:$ 
$$\text{
For every $t \in Q $, 
we have $\nd{F|_{X_{t}},h|_{X_{t}}} \geq \nd{F|_{X_{t_{0}}},h|_{X_{t_{0}}}}$. 
}
$$
Here $(F|_{X_{t}},h|_{X_{t}})$ denotes 
the singular hermitian line bundle restricted to the fiber $X_{t}$ at $t$ 
and $\nd{F|_{X_{t}},h|_{X_{t}}}$ denotes its numerical Kodaira dimension. 
$($See Definition \ref{nd} for the precise definition$)$.
\end{prop}

By combining the celebrated vanishing theorem proved in \cite{Cao14} and 
the (strong) openness theorem proved in \cite{GZ13} with the proof of Proposition \ref{lsc}, 
we obtain a relative vanishing theorem of 
Kawamata-Viehweg-Nadel type. 

\begin{theo}[Relative vanishing theorem of Kawamata-Viehweg-Nadel type] \label{KVN}
\ \\
Let $\pi \colon X \to \Delta$ be a surjective projective morphism 
from a complex manifold $X$ to an analytic space $\Delta$, 
and $(F,h)$ be a $($possibly$)$ singular hermitian line bundle on $X$ 
with semi-positive curvature.

Then we have 
$$
R^q \pi_{*}(K_{X}\otimes F \otimes \I{h})=0
\quad \text{for every $q>f - 
\max_{\substack{\pi \text{ is smooth} \\ \text{ at }   t \in \Delta }}$} 
\nd{F|_{X_{t}},h|_{X_{t}}}, 
$$
where $f$ is the dimension of general fibers. 
In particular, 
if $(F|_{X_{t}},h|_{X_{t}})$ is big 
for some point $t$ in the smooth locus of $\pi$, 
then we have 
$$
R^q \pi_{*}(K_{X}\otimes F \otimes \I{h})=0
\quad \text{for every $q>0$}. 
$$
\end{theo}

Moreover, we obtain an extension theorem (see Corollary \ref{cor-2}),  
which is motivated by the following problem 
related to the invariance of plurigenera 
and the dlt extension conjecture in the minimal model program 
(see \cite{Lev83}, \cite{Siu98}, \cite{Siu02}, \cite{Taka97}, \cite{Pau07} for 
the invariance of plurigenera, 
\cite{dhp}, \cite{FG14}, \cite{GM13} 
for the dlt extension conjecture,  and the references therein).

\begin{prob}\label{prob}
Let $\pi \colon X \to \Delta$ be 
a surjective proper K\"ahler morphism 
from a complex manifold $X$ to 
an open disk $\Delta \subset \mathbb{C}$. 
Assume that $K_{X}$ is $\pi$-nef and the central fiber $X_{0}:=\pi^{-1}(0)$ 
is simple normal crossing. 
Then can we extend a section 
$u \in H^{0}(X_{0}, \mathcal{O}_{X_{0}}(K_{X}^{m}))$ 
to a section in $ H^{0}(X, \mathcal{O}_{X}(K_{X}^{m}))$ 
$($by replacing $\Delta$ with a smaller disk if necessary$)$$?$ 
\end{prob}

The formulation of the above problem seems to be reasonable, 
since it can be seen as a relative version of the dlt extension conjecture 
and follows from the abundance conjecture. 
When $\pi \colon X \to \Delta$ is a smooth projective morphism, 
the complete answer 
was given in \cite{Siu98} and \cite{Siu02} 
(see also \cite{Pau07}). 
When the central fiber $X_{0}$ is smooth and 
$\mathcal{O}_{X_{0}}(K_{X})$ is semi-ample, 
Problem \ref{prob} was affirmatively solved in \cite{Lev83}. 
It was also shown in \cite{Taka97} 
that every section of the pluri-canonical bundle 
on each component of $X_{0}$ 
can be extended. 
The following result asserts that 
the above problem is affirmative 
if $K_{X}$ admits a singular metric  
with mild singularities. 
This result is a relative version of \cite[Theorem 1.4]{GM13} and 
it can be seen as an improvement of \cite{Lev83}.

\begin{cor}\label{cor-2}
Under the same situation as in Problem \ref{prob}, 
further let $(F,h)$ be a singular hermitian line bundle on $X$ 
with semi-positive curvature. 
Then every section in 
$H^{0}(X_{0}, \mathcal{O}_{X_{0}}(K_{X}\otimes F))$ 
that comes from $H^{0}(X_{0}, \mathcal{O}_{X_{0}}(K_{X}\otimes F)\otimes \I{h})$ can be extended to a section in 
$H^{0}(X, \mathcal{O}_{X}(K_{X}\otimes F)\otimes \I{h})$ 
by replacing $\Delta$ with a smaller disk. 
In particular, 
if $K_{X}$ admits a singular metric $h$ whose curvature 
is semi-positive and Lelong number is zero at every point in $ X_{0}$, 
then Problem \ref{prob} is affirmatively solved. 
\end{cor}

At the end of this section, 
we briefly explain the sketch of the proof of Theorems \ref{main}, 
comparing with the absolute case established in \cite{Mat13}. 
The proof of Theorem \ref{main2} is essentially the same as in Theorem \ref{main}. 
For the proof of the main results, 
we generalize methods in \cite{Tak95} and \cite{Fuj13},  
and combine them with techniques in \cite{FM16} and \cite{Mat13}.

In Step \ref{S1}, 
we approximate a given singular metric $h$ 
by singular metrics $\{ h_{\e} \}_{\e>0}$ 
that are smooth on a Zariski open set $Y_{\e}$, 
which enables us to use the theory of harmonic integrals on $Y_{\e}$. 
Note that we can not directly use the theory of harmonic integrals 
since $h$ may have non-algebraic  singularities. 
The subvariety $Y_{\e}$ is independent of $\e$ in \cite{Mat13},  
but $Y_{\e}$ may essentially depend on $\e $ in our case.  
For this reason, 
we construct a complete K\"ahler form $\omega_{\e, \delta}$ 
on $Y_{\e}$ 
such that $\omega_{\e, \delta}$ converges to 
a K\"ahler metric $\omega$ on $X$ as $\delta$ tends to zero. 
By the standard De Rham-Weil isomorphism, 
we can represent a given cohomology class $\{u\}$ by 
an $F$-valued differential form $u$. 
The $F$-valued form $u$ is locally $L^2$-integrable, 
but unfortunately $u$ may not be $L^2$-integrable on $X$
due to the non-compactness of $X$. 
For this reason, 
in Step \ref{S2}, we  construct a new metric $H_{\e}$ on $F$ 
by suitably choosing an exhaustive plurisubharmonic function on $X$, 
which enables us to take 
harmonic $L^{2}$-forms $u_{\e, \delta}$ with respect to $H_{\e}$ 
and $\omega_{\e, \delta}$ representing the same cohomology class $\{u\}$.   
In Step \ref{S3}, we reduce the proof 
to show that 
the $L^2$-norm $\|su_{\e, \delta}\|_{X_{c}, H_{\e}, \omega_{\e, \delta}}$ 
on a relatively compact set $X_{c} \Subset X$ 
converges to zero as $\e \to 0$ and $\delta \to 0$. 
For this step, we need that the quotient map from the $L^{2}$-space 
to the $\dbar$-cohomology group is a compact operator 
(see Proposition \ref{DWiso}). 
In Step \ref{S4},  
we construct a solution $v_{\e, \delta}$ of 
the $\dbar$-equation $\dbar v_{\e, \delta} = su_{\e, \delta}$ 
such that   
the $L^{2}$-norm $\| v_{\e, \delta} \|_{X_{c}, H_{\e},\omega_{\e, \delta}}$ on $X_{c}$  
is uniformly bounded, 
by using the $\rm{\check{C}}$ech complex and the De Rham-Weil isomorphism. 
Finally  we prove that 
\begin{align*}
\| su_{\e, \delta} \|_{X_{c}, H_{\e}, \omega_{\e, \delta}} ^{2} &= 
\lla su_{\e,\delta}, \dbar v_{\e, \delta} \rra_{X_{c}, H_{\e}, \omega_{\e, \delta}} \\
&=\lla \dbar^{*} su_{\e, \delta}, v_{\e,\delta} \rra_{X_{c}, H_{\e}, \omega_{\e, \delta}} 
+\la (d\Phi)^{*} su_{\e,\delta}, v_{\e,\delta} \ra_{\partial X_{c}, H_{\e},\omega_{\e, \delta}}
\end{align*}
for almost all $X_{c}\Subset X$, 
by generalizing the formula in \cite{FK} (see Proposition \ref{key}). 
Here ${\bullet}^{*}$  denotes 
the adjoint operator of $\bullet$ 
and $\la \bullet, \bullet \ra_{\partial X_{c}}$ denotes the norm 
on the boundary $\partial X_{c}$. 
We remark that the norm on the boundary $\partial X_{c}$ appears 
due to the non-compactness of $X$. 
In Step \ref{S5}, we show that 
the norm of $\dbar^{*} su_{\e, \delta}$ and 
$(d\Phi)^{*} su_{\e,\delta}$ 
converges to zero 
by using Ohsawa-Takegoshi's twisted Bochner-Kodaira-Nakano identity.

This paper is organized as follows\,$:$  
In Section \ref{Sec2}, 
we summarize the results needed in this paper.  
Moreover, in this section, 
we give a generalization of \cite[(1.3.2) Proposition]{FK} 
and recall the fundamental facts on the construction of the De Rham-Weil isomorphism 
in our situation.  
In Section \ref{Sec3}, we prove Theorem \ref{main} and Theorem \ref{main2}. 
In Section \ref{Sec4}, we prove Corollary \ref{cor-1}, Proposition \ref{lsc}, 
Theorem \ref{KVN}, and Corollary \ref{cor-2}.

\subsection*{Acknowledgements}
The author wishes to express his gratitude to Professor Osamu
Fujino for giving several remarks on the proof for the case $m=0$ of Theorem \ref{main} 
and for suggesting him to consider Theorem \ref{KVN}. 
He is supported by the Grant-in-Aid for 
Young Scientists (B) $\sharp$25800051 from JSPS.

\section{Preliminaries}\label{Sec2}
In this section, 
we summarize the results used in this paper with our notation. 
Throughout this section, 
let $X$ be a complex manifold of dimension $n$ 
and $F$ be a (holomorphic) line bundle on $X$.

\subsection{$L^2$-spaces of differential forms}\label{Sec2-1}
In this subsection, 
we recall $L^2$-spaces of $F$-valued differential forms and 
operators defined on them. 
Let $\omega$ be a positive $(1,1)$-form on $X$ 
and $h$ be a smooth (hermitian) metric on $F$. 

For $F$-valued $(p,q)$-forms $u$ and $v$,  
the (global) inner product 
$\lla u, v \rra  _{h, \omega}$ 
is defined by
\begin{equation*}
\lla u, v \rra  _{h, \omega}:=
\int_{X} 
\langle u, v\rangle _{h, \omega}\, dV_{\omega}, 
\end{equation*}
where $dV_{\omega}$ is the volume form defined by $dV_{\omega}:=\omega^{n}/n!$ 
and $\langle u, v\rangle _{h, \omega}$ is the point-wise inner product 
with respect to $h$ and $\omega$.  
The $L^{2}$-space of $F$-valued $(p, q)$-forms with respect to 
$h$ and $\omega$ is defined by
\begin{align*}
L_{(2)}^{p, q}(X, F)_{h, \omega}&:= 
\{u  \,|\,  u \text{ is an }F\text{-valued }(p, q)\text{-form with } 
\|u \|_{h, \omega}< \infty. \}. 
\end{align*}

The Chern connection $D=D_{(F, h)}$ on $F$  
is canonically determined by the holomorphic structure and 
the smooth metric $h$ on $F$, 
which can be written as 
$D = D'_{h} + D''_{h}$ 
with the $(1,0)$-connection $D'_{h}$ and the 
$(0,1)$-connection $D''_{h}$.
We remark that $D''_{h}$ agrees with the $\overline{\partial}$-operator. 
The connections $\dbar$ and $D'_{h}$   
(strictly speaking, their maximal extension) 
can be seen as a densely defined closed operator    
on $L_{(2)}^{p, q}(X, F)_{h, \omega}$ 
with the following domain\,$:$ 
$$
{\rm{Dom}}\, \dbar:=\{u \in L_{(2)}^{p, q}(X, F)_{h, \omega} \, | \, 
\dbar u \in L_{(2)}^{p, q+1}(X, F)_{h, \omega} \}. 
$$ 
Strictly speaking, these operators depend on $h$ and $\omega$
since their domain and range depend on them, 
but we often omit the subscript  
(for example, we abbreviate $\dbar_{h,\omega}$ to $\dbar$).

We consider the Hodge star operator $\ast$
with respect to $\omega$
$$
\ast=\ast_{\omega}\, \colon \, C_{\infty}^{p,q}(X,F) \to 
C_{\infty}^{n-q,n-p}(X,F),  
$$
where $C_{\infty}^{p,q}(X,F)$ is 
the set of smooth $F$-valued $(p,q)$-forms on $X$. 
By the definition, we have 
$\langle u, v\rangle _{h, \omega}\, dV_{\omega}=
u \wedge H \overline{\ast v}$, 
where $H$ is a local function representing $h$. 
In this paper, the notation $A ^{*}$ denotes 
the formal adjoint of an operator $A$. 
For example ${D_{h,\omega}'^{*}}$ and $\dbar^{*}_{h,\omega}$ 
are respectively the formal adjoint operator of $D_{h,\omega}'$ and $\dbar$.  
We remark that 
$$
{D_{h,\omega}'^{*}}=-\ast \dbar \ast \quad \text{and} \quad 
\dbar^{*}_{h,\omega}=-\ast {D'}_{h,\omega} \ast. 
$$
Further, for a differential form $\theta$, 
the notation $\theta^{\ast}$ denotes the adjoint operator 
with respect to the point-wise inner product. 
When $\theta$ is of type $(s,t)$ and 
$\theta^{\ast}$ acts on $C_{\infty}^{p,q}(X,F)$, 
we have 
$$
\theta^{\ast}= (-1)^{(p+q)(s+t+1)}{\ast} {\overline{\theta}} {\ast}. 
$$
For operators $A$ and $B$ with pure degree, 
the graded bracket $[\bullet, \bullet]$ is defined by 
$$
[A,B]:=AB-(-1)^{\deg A \deg B}BA.
$$

In the proof of the main results, 
we often use the following lemmas, 
which are obtained from simple computations 
(for example, see \cite{Tak95} for Lemma \ref{koukan}).

\begin{lemm}\label{koukan}
If $\omega$ is a K\"ahler form, then we have the following identities\,$:$ 
\begin{itemize}
\item[$\bullet$] $\theta^{*}=\sqrt{-1}[\overline{\theta}, \Lambda_{\omega}]$
for a $(1,0)$-form $\theta$. 
\item  $\eta^{*}=-\sqrt{-1}[\overline{\eta}, \Lambda_{\omega}]$
for a $(0,1)$-form $\eta$. 
\item[$\bullet$] $[\dbar, (\dbar \Phi)^{*}]+[{D_{h,\omega}'^{*}}, \partial \Phi]
=[\deldel \Phi, \Lambda_{\omega}]$
for a smooth function $\Phi$. 
\end{itemize}
Here $\Lambda_{\omega}$ denotes the adjoint operator 
defined by $\Lambda_{\omega}:=\omega^{\ast}$. 
\end{lemm}

\begin{lemm}\label{mul}
Let $\ome$ and $\omega$ be positive $(1,1)$-forms such that $\ome \geq \omega$. 
Then we have the following\,$:$
\begin{itemize}
\item[$\bullet$] There exists $C>0$ such that 
$|\theta^{*}u|_{\omega}\leq C |\theta|_{\omega}|u|_{\omega}$ 
for differential forms $\theta$, $u$.
\item[$\bullet$] The inequality $|\theta|_{\ome} \leq |\theta|_{\omega}$ holds for a differential form $\theta$. 
\item[$\bullet$] The inequality $|\theta|_{\ome}\, dV_{\ome} 
\leq |\theta|_{\omega}\,  dV_{\omega}$ holds
for an $(n,q)$-form $\theta$. 
\item[$\bullet$] The equality $|\theta|_{\ome}\, dV_{\ome} 
= |\theta|_{\omega}\,  dV_{\omega}$ holds
for an $(n,0)$-form $\theta$. 
\end{itemize}
\end{lemm}
\begin{proof}
For a given point $x \in X$, we choose  
a local coordinate $(z_{1}, z_{2}, \dots, z_{n})$ such that 
\begin{align*}
\ome = \frac{\sqrt{-1}}{2} \sum_{j=1}^{n} 
\lambda_{j} dz_{j} \wedge d\overline{z_{j}} \quad \text{and} \quad 
\omega = \frac{\sqrt{-1}}{2} \sum_{j=1}^{n} 
 dz_{j} \wedge d\overline{z_{j}}
\quad  {\rm{ at}}\ x. 
\end{align*}
When differential forms $\theta$ and $u$ are written as 
$\theta=\sum_{I, J}\theta_{I,J}dz_{I} \wedge d\overline{z}_{J}$ 
and 
$u=\sum_{K, L}u_{K,L}dz_{K} \wedge d\overline{z}_{L}$ 
in terms of this coordinate, 
we have  
$$|\theta|^2_{\omega}=\sum_{I, J}|\theta_{I,J}|^{2} 
\quad \text{and} \quad  
|\theta|^2_{\ome}=\sum_{I, J}|\theta_{I,J}|^{2} 
\frac{1}{\prod_{(i,j) \in (I, J)} \lambda_{i} \lambda_{j}}
\quad  {\rm{ at}}\ x, 
$$ 
where $I$, $J$, $K$, $L$  are ordered multi-indices. 
The second claim follows from $\lambda_{i} \geq 1$. 
Further 
we can easily check the third claim and the fourth claim from the above equalities. 
By $|\theta_{I,J}| \leq |\theta|_{\omega}$ and $|u_{K,L}| \leq |u|_{\omega}$,  
we have  
\begin{align*}
| \theta \wedge u |_{\omega} 
& =
\big| \sum_{I, J, K, L} 
\theta_{I,J}\, u_{K,L}\, dz_{I} \wedge d\overline{z}_{J}
\wedge dz_{K} \wedge d\overline{z}_{L} \, \big|_{\omega} \\
& \leq \sum_{I, J, K, L} |\theta_{I,J}|  |u_{K,L}| \leq \sum_{I, J, K, L}  |\theta|_{\omega} |u|_{\omega}=C_{1} |\theta|_{\omega} |u|_{\omega}.   
\end{align*}
Here $C_{1}$ is a positive constant 
depending only on the degree of differential forms. 
On the other hand, since the Hodge star operator $*$ preserves 
the point-wise norm, we have 
$$|\theta^{*}u|_{\omega}=|\ast \overline{\theta} \ast u|_{\omega}=
|\overline{\theta} \ast u|_{\omega}\leq C_{2}|\overline{\theta}|_{\omega}\, |\ast u|_{\omega}
= C_{2}|\theta|_{\omega} |u|_{\omega}
$$
for some constant $C_{2}>0$.
\end{proof}

\subsection{Ohsawa-Takegoshi's twisted 
Bochner-Kodaira-Nakano identity}\label{Sec2-2}

The following proposition is obtained from 
the twisted Bochner-Kodaira-Nakano identity 
(cf. \cite{DF83}, \cite{DX84}, \cite{OT87}, \cite{Ohs95}), 
which plays an important role in Step \ref{S5}. 
For example, see \cite[Theorem 2.2]{Tak95} for the precise proof.

\begin{prop}[Twisted Bochner-Kodaira-Nakano identity]\label{Nak}
Let $\omega$ be a complete K\"ahler form on $X$ 
such that  $\sqrt{-1}\Theta_{h}(F)\geq -C_{1} \omega $ 
for some constant $C_{1}$. 
Further let $\Phi$ be a bounded smooth function on $X$ 
such that $\sup_{X}|d\Phi|_{\omega}<\infty$ and 
$\deldel \Phi \geq -C_{2}\omega$ for some constant $C_{2}$.

Then, for every $u \in {\rm{Dom}}\, \dbar^{*}_{h, \omega} \cap {\rm{Dom}}\, \dbar 
\subset L^{n,q}_{(2)}(X,F)_{h,\omega}$, 
we have 
\begin{align*}
&\|\sqrt{\eta}(\dbar+\dbar \Phi) u\|^2_{h, \omega}+
\|\sqrt{\eta}\dbar^{*}_{h, \omega}u \|^2_{h, \omega}\\
=&
\|\sqrt{\eta}({D_{h,\omega}'^{*}}-(\partial \Phi)^{*}) u\|^2_{h, \omega}
+\lla \eta \sqrt{-1} (\Theta_{h}(F) +\partial \dbar \Phi)\Lambda_{\omega} u, u\rra_{h, \omega}\, ,   
\end{align*}
where $\eta$ is the function defined by $\eta:=e^{\Phi}$.
\end{prop}

We remark that the case $\Phi \equiv 0$ corresponds to the non-twisted version.   
In the proof of Proposition \ref{key} (also Proposition \ref{Nak}), 
we use the following lemma due to  Andreotti-Vesentini 
(see \cite{AV65}, \cite[LEMME 4.3]{Dem82}, \cite{Ves67}).

\begin{lemm}[Density lemma]\label{density}
Let $\omega$ be a complete positive $(1,1)$-form on $X$. 
\vspace{0.1cm}
\\
$\bullet$
There exists a sequence of cut-off functions $\{\theta_{k} \}_{k=1}^{\infty}$ on $X$ 
such that ${\rm{Supp}}\, \theta_{k} \Subset X$, $|d \theta _{k}|_{\omega} \leq 1$, and 
that $\theta_{k} \to 1$ as $k \to \infty$. 
\vspace{0.1cm}\\
$\bullet$
The set of smooth $F$-valued $(p,q)$-forms with compact support 
is dense in ${\rm{Dom}}\, \dbar^{*}_{h, \omega}$, ${\rm{Dom}}\, \dbar$, and  
${\rm{Dom}}\, \dbar^{*}_{h, \omega} \cap {\rm{Dom}}\, \dbar$ respectively  
with respect to the following graph norms\,$:$ 
\begin{align*}
\|u\|_{h,\omega}+\|\dbar^{*}_{h, \omega} u\|_{h,\omega}, \quad 
\|u\|_{h,\omega}+\|\dbar u\|_{h,\omega},  \quad  \text{and}  \quad    
\|u\|_{h,\omega}+\|\dbar^{*}_{h, \omega} u\|_{h,\omega} +  \|\dbar u\|_{h,\omega}. 
\end{align*}

\end{lemm}

\subsection{Adjoint operators on domains with boundaries}\label{Sec2-3}

Let $\Phi$ be a smooth function on $X$. 
In this subsection, we consider the level set $X_{c}$ defined by 
$X_{c}:= \{x\in X\,  | \, \Phi(x) < c \}$ 
such that 
$X_{c} \Subset X$ and 
$d\Phi \not= 0$ on the boundary $\partial X_{c}$ of $X_{c}$. 
The inner product on the boundary $\partial X_{c}$ 
is defined to be  
$$
\la u, v \ra_{\partial X_{c}, h,\omega}:=\int_{\partial X_{c}} 
\langle u, v\rangle_{h, \omega} \, dS_{\omega} 
$$
for $F$-valued $(p,q)$-forms $u$, $v$ 
that are smooth on a neighborhood of $\partial X_{c}$. 
Here $dS_{\omega}$ denotes the volume form on $\partial X_{c}$ 
defined by $dS_{\omega}:=\ast d\Phi / |d \Phi|^2_{\omega}$. 
Note that we have $dV_{\omega}=d \Phi \wedge dS_{\omega}$ by the definition. 
Then Stoke's theorem yields  
$$
\lla \dbar u, v\rra_{X_{c}, h, \omega}=
\lla u, \dbar^{*}_{h, \omega} v\rra_{X_{c}, h, \omega} + 
\la u, (\dbar \Phi)^{*} v\ra_{\partial X_{c}, h, \omega}
$$
for a smooth $F$-valued $(p,q-1)$-form $u$ and 
a $(p,q)$-form $v$ on $X$
(see \cite[(1.3.2) Proposition]{FK}).

For our purposes, we need to generalize the above formula to 
a Zariski open set $Y \subset X$ equipped with  
a complete positive $(1,1)$-form $\ome$.  
In the following proposition, 
we consider the Hodge star operator $\ast:=\ast_{\ome}$, 
the volume form $dS_{\ome}:=\ast_{\ome} d\Phi / |d \Phi|^2_{\ome}$, 
the inner product 
$\la  \dbar u, v\ra_{\partial X_{c}, h, \ome}:=\int_{\partial X_{c} \cap Y} 
\langle u, v\rangle_{h, \ome} \, dS_{\ome}$, and so on  
with respect to $\ome$ (not $\omega$).

\begin{prop}\label{key}
Let $\ome$ be a complete positive $(1,1)$-form 
on a Zariski open set $Y$ of a complex manifold $X$, 
and $u$ $($resp. $v$$)$ be a smooth $F$-valued $(p,q-1)$-form $($resp. $(p,q)$-form$)$ on $Y$ 
with the finite $L^{2}$-norms $\|u\|_{h,\ome}$, $\|v\|_{h,\ome}$, 
$\|\dbar u\|_{h,\ome}$, $\|\dbar^{*}_{h,\ome} v\|_{h,\ome}< \infty$. 
Consider a smooth function $\Phi$ on $X$ and 
the level set $X_{d}$ defined by 
$X_{d}:= \{x\in X\,  | \, \Phi(x) < d \}$ 
for $d \in \mathbb{R}$. 
If $X_{c} \Subset X$ and 
$d\Phi \not= 0$ on the boundary $\partial X_{c}$ of $X_{c}$ for some $c \in \mathbb{R}$, 
then there exists a sufficiently small number $a>0$ with the following properties\,$:$
\begin{itemize}
 \item[$\bullet$] $d\Phi \not = 0$ on $\partial X_{d}$ 
for every $d\in (c-a, c+a)$. 
\item[$\bullet$] $\lla \dbar u, v\rra_{X_{d}, h, \ome}=
\lla u, \dbar^{*}_{h,\ome} v\rra_{X_{d},h, \ome} + 
\la u, (\dbar \Phi)^{*} v\ra_{\partial X_{d}, h, \ome} $\, 
for almost all $d\in (c-a, c+a)$.
\end{itemize}
\end{prop}

\begin{rem}\label{key-rem}
(1) In the case of $Y=X$, the above equality 
in the second property holds for arbitrary 
$d \in (c-a, c+a)$  
(see \cite[(1.3.2) Proposition]{FK}). 
Proposition \ref{key} can be seen as a generalization of this result 
to Zariski open sets. 
\vspace{0.1cm}\\ 
(2) By the proof, we see that $a$ depends only on $\Phi$, 
but does not depend on $u$, $v$, and $\ome$. 
\vspace{0.1cm}\\ 
(3) In the proof of the main results, 
we apply Proposition \ref{key}  
to a family of countably many differential forms. 
The subset $I$ defined by 
\begin{equation*}
I :=\{d \in (c-a, c+a) \,  |  \, 
\text{The above equality does \lq \lq not" hold for $d$. }\}
\end{equation*}
depends on $u$, $v$, $\ome$. 
The Lebesgue measure of $I$ is zero by the proposition,   
and a countable union of subsets of zero Lebesgue measure 
also has Lebesgue measure zero. 
Therefore, for given countably many differential forms,  
there exists a common $I$ of zero Lebesgue measure 
such that the above equality in the second property holds for $d \not \in I$. 
\end{rem}

\begin{proof}
For simplicity, we consider only the case 
where $(F, h)$ is trivial. 
For a sufficiently small $a>0$, 
we have $d\Phi \not = 0$ on $\partial X_{d}$ 
for every $d \in (c-a, c+a)$  
by the assumption $d \Phi \not = 0$ on ${\partial X_{c}}$. 
For the proof of the second property,  
we take a sequence of cut-off functions 
$\{\theta_{k} \}_{k=1}^{\infty}$ on $Y$ 
such that ${\rm{Supp}}\, \theta_{k} \Subset Y$ and 
$\theta_{k} \to 1$ as $k \to \infty$. 
Since $\ome$ is complete on $Y$, 
we can add the property $|d \theta _{k}|_{\ome} \leq 1$ 
(see Lemma \ref{density}). 
Then Stoke's theorem implies 
\begin{align}\label{stokes}
\lla \dbar  (\theta _{k} u),v\rra_{X_{d}, \ome}=
\lla  \theta _{k}u, \dbar^{*}_{\ome} v\rra_{X_{d}, \ome} + 
\la  \theta _{k}u, (\dbar \Phi)^{*} v\ra_{\partial X_{d}, \ome}.  
\end{align}
We remark that all integrals and adjoint operators are 
computed with respect to $\ome$ (not $\omega$). 
There is no difficulty in proving equality (\ref{stokes}) 
since all integrands that appear in equality (\ref{stokes})  are zero 
on a neighborhood of the subvariety $X\setminus Y$. 
Indeed, we have    
$$
\lla \dbar  (\theta _{k} u),v\rra_{X_{d}, \ome}=
\lla  \theta _{k}u, \dbar^{*}_{\ome} v\rra_{X_{d}, \ome} + 
\int_{\partial X_{d}} \theta _{k}u \wedge \overline{\ast v}  
$$
by $(d\theta _{k}u) \wedge \overline{\ast v}=
-\theta _{k}u \wedge (\overline{\ast \ast d \ast v})+
d(\theta _{k}u \wedge \overline{\ast v})$ and Stoke's theorem. 
Further we have   
\begin{align*}
\big\{ u \wedge \overline{\ast v} \big\}  \wedge d \Phi
= - u \wedge \overline{\ast \ast d \Phi \wedge \ast v}
= -\langle u, (d\Phi)^{\ast} v \rangle\, dV_{\ome}
= \big\{  \langle u, (d\Phi)^{\ast} v \rangle\, dS_{\ome} \big\}\wedge d\Phi.  
\end{align*}
Moreover it follows that 
$d \Phi$ is non-zero in the normal direction of $\partial X_{d}$
from $d \Phi |_{\partial X_{d}} = 0$ and $d \Phi \not=0$ on $\partial X_{d}$. 
Therefore we can conclude that 
$\int_{\partial X_{d}} u \wedge \overline{\ast v}=
\int_{\partial X_{d}} \langle u, (d\Phi)^{\ast} v \rangle\, dS_{\ome}$, 
and thus we obtain equality (\ref{stokes}).

Now we observe the limit of each term. 
By the bounded Lebesgue convergence theorem, 
$\theta _{k}u$ converges to $u$ as $k \to \infty$
in the $L^2$-topology with respect to $\ome$. 
Here we used the assumption $\|u\|_{\ome}<\infty$. 
On the other hand, from 
$\dbar  (\theta _{k} u)=\dbar \theta _{k} \wedge u +\theta _{k} \dbar u$ and $|d \theta _{k}|_{\ome} \leq 1$ and $d \theta _{k} \to 0$ 
in the point-wise sense,   
we can easily see that $\dbar \theta _{k} \wedge u \to 0$ and 
$\theta _{k} \dbar u \to \dbar u$ 
in the $L^2$-topology, 
by using the bounded Lebesgue convergence theorem again. 
Here we used the assumption $ \|\dbar u\|_{\ome}<\infty$. 
From the above argument,  the left hand side 
(resp. the first term of the right hand side) in equality (\ref{stokes})  
converges to $\lla \dbar u,v\rra_{X_{d}, \ome}$
(resp. $\lla u, \dbar^{*}_{\ome} v\rra_{X_{d}, \ome}$) 
for every $d$. 

It remains to show that 
the second term of the right hand side in equality (\ref{stokes})  converges to 
$\la u, (\dbar \Phi)^{*} v\ra_{\partial X_{d}, \ome}$ for almost all $d$. 
By Cauchy-Schwarz's inequality and Lemma \ref{mul}, 
the integrand of the second term can be estimated as follows$:$
\begin{align*}
|\langle \theta _{k}u, (\dbar \Phi)^{*} v \rangle_{\ome}|
\leq | \theta _{k}u|_{\ome}\,  |(\dbar \Phi)^{*} v|_{\ome}
\leq  C\sup_{\partial X_{d}} \big(|\dbar \Phi|_{\ome}\big)\,   
| u|_{\ome}\, |v|_{\ome}. 
\end{align*}
We may assume $\ome \geq \overline{\omega}$ 
for some positive $(1,1)$ form $\overline{\omega}$ on $X$ 
since $\ome$ is complete on $Y$. 
Then the inequality $|\dbar \Phi|_{\ome} 
\leq |\dbar \Phi|_{\overline{\omega}}$ holds by Lemma \ref{mul}. 
In particular, the sup-norm $\sup_{\partial X_{d}}|\dbar \Phi|_{\ome}$ is finite 
since the function $\Phi$ is smooth on $X$ (not $Y$). 
If the integral of 
$ | u|_{\ome}\, |v|_{\ome}$ 
on $\partial X_{d}$ is finite,  
the integral 
$\la  \theta _{k}u, (\dbar \Phi)^{*} v\ra_{\partial X_{d}, \ome}$ 
converges to $\la u, (\dbar \Phi)^{*} v\ra_{\partial X_{d}, \ome}$
by the bounded Lebesgue convergence theorem. 
In order to prove that the integral of 
$ | u|_{\ome}\, |v|_{\ome}$ on $\partial X_{d}$ 
is finite  for almost all $d$, 
we apply Fubini's theorem and H\"older's inequality. 
Then we obtain  
\begin{align*}
\int_{d \in (c-a, c+a)}
\Big( \int_{\partial X_{d}} | u|_{\ome} |v|_{\ome} 
\, dS_{\ome} \Big) d\Phi 
&=
\int_{\{c-a<\Phi<c+a\}}\hspace{-0.9cm}
 | u|_{\ome} |v|_{\ome}\,  dV_{\ome}\\
&\leq \Big( \int_{\{c-a<\Phi<c+a\}}\hspace{-0.9cm}
 | u|^2_{\ome}\,   dV_{\ome} \Big)^{1/2} 
\Big( \int_{\{c-a<\Phi<c+a\}}\hspace{-0.9cm}
 | v|^2_{\ome}\,   dV_{\ome} \Big)^{1/2}.
\end{align*}
Here we used the equality $ d\Phi \wedge dS_{\ome} = dV_{\ome}$. 
Since the right hand side is finite thanks 
to the assumptions $\|u\|_{\ome}, \|v\|_{\ome} < \infty$, the integral 
$\int_{\partial X_{d}} | u|_{\ome} |v|_{\ome} \, dS_{\ome}$ 
should be finite for almost all $d$ in $(c-a, c+a)$.
(Otherwise, the left hand side becomes infinity.)  
This completes the proof.
\end{proof}

\subsection{Singular hermitian metrics and multiplier ideal sheaves}
\label{Sec2-4}

We recall the definition 
of singular hermitian metrics and curvatures. 
Fix a smooth (hermitian) metric $g$ on $F$. 

\begin{defi}\label{s-met}
{(Singular hermitian metrics and curvatures).}
(1) For an $L^{1}_{\rm{loc}}$-function $\varphi$ on a complex manifold $X$, 
the hermitian metric $h$ defined by 
\begin{equation*}
h:= g e^{-2\varphi} 
\end{equation*}
is called a \textit{singular hermitian metric} on $F$. 
Further $\varphi$ is called the \textit{weight} of $F$ 
with respect to the fixed smooth metric $g$. 
\vspace{0.2cm} \\ 
(2) The {\textit{curvature}} 
$\sqrt{-1} \Theta_{h}(F)$ 
associated to $h$ is defined by   
\begin{equation*}
\sqrt{-1} \Theta_{h}(F) := \sqrt{-1} \Theta_{g}(F)+ 2\deldel \varphi, 
\end{equation*}
where $  \sqrt{-1} \Theta_{g}(F)$ is 
the Chern curvature of $g$. 
\end{defi}

For simplicity, we call the singular hermitian metric as the singular metric.
In this paper, we consider only a singular metric $h$ such that 
$\sqrt{-1} \Theta_{h}(F) \geq \gamma$ holds 
for some smooth $(1,1)$-form $\gamma$ on $X$.
Then the weight function $\varphi$ becomes 
a quasi-plurisubharmonic (quasi-psh for short) function.  
In particular, the function $\varphi$ is upper semi-continuous. 
Moreover, then, the multiplier ideal sheaf defined below 
is coherent by a theorem of Nadel.

\begin{defi}[Multiplier ideal sheaves]\label{multiplier}
Let $h$ be a singular metric on $F$ such that 
$\sqrt{-1} \Theta_{h}(F) \geq \gamma$ for some smooth $(1,1)$-form $\gamma$ on $X$. 
The ideal sheaf $\I{h}$ defined 
to be  
\begin{equation*}
\I{h}(B):= \I{\varphi}(B):= \{f \in \mathcal{O}_{X}(B)\ \ |\ \  
|f|\, e^{-\varphi} \in L^{2}_{\rm{loc}}(B) \}
\end{equation*}
for every open set $B \subset X$, is called 
the \textit{multiplier ideal sheaf} associated to $h$. 
\end{defi}

In Step \ref{S1}, 
we approximate a given singular metric by singular metrics 
that are smooth on a Zariski open set. 
The following theorem is a reformulation of 
the equisingular approximation, 
which is proved by a slight revision of 
the proof of \cite[Theorem 2.3]{DPS01}. 

\begin{theo}[{\cite[Theorem 2.3]{DPS01}}] \label{equi}
Let $\omega$ be a positive $(1,1)$-form on a complex manifold $X$ 
and $(F, h)$ be a singular hermitian line bundle on $X$. 
Assume that $\sqrt{-1}\Theta_{h}(F) \geq \gamma$ holds 
for a smooth $(1,1)$-form $\gamma $ on $X$. 
Then, for a relatively compact set $K \Subset X$, 
there exist singular metrics $\{h_{\e} \}_{1\gg \e>0}$ 
on $F|_{K}$ with the following properties\,$:$
\begin{itemize}
\item[(a)] $h_{\e}$ is smooth on $K \setminus Z_{\e}$, 
where $Z_{\e}$ is a proper subvariety of $K$.
\item[(b)]$h_{\e''} \leq h_{\e'} \leq h$ holds on $K$ 
for any $0< \e' < \e'' $.
\item[(c)]$\I{h}= \I{h_{\e}}$ on $K$.
\item[(d)]$\sqrt{-1} \Theta_{h_{\e}}(F) \geq \gamma -\e \omega$ on $K$. 
\end{itemize}
\end{theo}

\begin{rem}\label{equi-rem}
For a complete positive $(1,1)$-form $\omega_{K}$ on $K$, 
we may assume that property (d) holds for $\omega_{K}$, 
that it, $\sqrt{-1} \Theta_{h_{\e}}(F) \geq \gamma -\e \omega_{K}$ holds on $K$. 
Indeed we have $\omega_{K} \geq a \omega$ for some small number $a>0$ 
since $\omega_{K}$ is complete, $\omega$ is defined on $X$, and 
$K$ is relatively compact in $X$. 
By applying Theorem \ref{equi} to $a \omega$, 
we can easily see that $\sqrt{-1} \Theta_{h_{\e}}(F) \geq \gamma -\e a \omega 
\geq \gamma -\e \omega_{K}$. 
\end{rem}

\begin{proof}
In \cite{DPS01}, 
the theorem has been proved 
in the case where $K=X$ and $X$ is compact.  
In their proof, 
we essentially use the assumption that $X$ is compact 
only when we take a finite cover of $X$ 
with inequality \cite[(2.1)]{DPS01}. 
Since $K$ is a relatively compact in $X$, 
for a given $\e>0$, we can take a finite open cover of $K$ 
by open balls $B_{i}:=\{|z^{(i)}|<r_{i}\}$ with a local coordinate $z^{(i)}$
satisfying inequality \cite[(2.1)]{DPS01}. 
Then we can easily see that 
the same argument works 
by suitably replacing $X$ in \cite{DPS01} with $K$.
Indeed, uniform estimates \cite[(2.4), (2.5), (2.6)]{DPS01} 
can be checked in the same way
since it is sufficient to consider only $B_{i}$ in this step. 
Further the function $\psi_{\e, \nu}$ can be defined 
on a neighborhood of $K$ by the fourth line in \cite[page 700]{DPS01}. 
Therefore we can repeat the same argument.  
\end{proof}

\subsection{Fr\'echet spaces}
\label{Sec2-5}
In this subsection, 
we summarize fundamental facts on 
Hilbert spaces and Fr\'echet spaces. 
We  give a proof of them for the reader's convenience.

\begin{lemm}\label{weak-clo}
Let $L$ be a closed subspace in a Hilbert space $\mathcal{H}$. 
Then $L$ is closed with respect to the weak topology of $\mathcal{H}$, 
that is, if a sequence $\{w_{k}\}_{k=1}^{\infty}$ in $L$ 
weakly converges to $w$, then the weak limit $w$ belongs to $L$. 
\end{lemm}
\begin{proof}
By the orthogonal decomposition, there exists a closed subspace $M$ 
satisfying $L=M^{\perp}$. 
Then we obtain that 
$
0=\lla w_{k}, v \rra_{\mathcal{H}}
\to \lla w, v \rra_{\mathcal{H}}
$
as $k \to \infty$ for every $v \in M$. 
Therefore we have $w \in M^{\perp}=L$. 
\end{proof}

\begin{lemm}\label{weak-hil}
Let $\varphi: \mathcal{H}_{1} \to \mathcal{H}_{2}$ be 
a bounded operator $($continuous linear map$)$ between Hilbert spaces
$\mathcal{H}_{1}$ and $\mathcal{H}_{2}$. 
If $\{w_{k}\}_{k=1}^{\infty}$ weakly converges to $w$ in $\mathcal{H}_{1}$, 
then $\{\varphi(w_{k})\}_{k=1}^{\infty}$ weakly converges to 
$\varphi(w)$ in $\mathcal{H}_{2}$. 
\end{lemm}
\begin{proof}
By taking the adjoint operator $\varphi^*$, 
we obtain
$$
\lla \varphi(w_{k}), v \rra_{\mathcal{H}_{2}}
= \lla w_{k}, \varphi^*(v) \rra_{\mathcal{H}_{1}}
\to \lla w, \varphi^*(v) \rra_{\mathcal{H}_{1}}
=\lla \varphi(w), v \rra_{\mathcal{H}_{2}} 
$$
for every $v \in \mathcal{H}_{2}$. 
This completes the proof. 
\end{proof}

We apply the following lemma for 
the compact operator in Proposition \ref{DWiso}. 

\begin{lemm}\label{comp-hil}
Let $\varphi: \mathcal{H} \to \mathcal{F}$ be 
a compact operator from a Hilbert space 
$\mathcal{H}$ to a Fr${\acute{e}}$chet space $\mathcal{F}$. 
Assume $\mathcal{F}$ is the inverse limit of Hilbert spaces. 
If $\{w_{k}\}_{k=1}^{\infty}$ weakly converges to $w$ in $\mathcal{H}$, 
then $\{\varphi(w_{k})\}_{k=1}^{\infty}$ 
converges to $\varphi(w)$ in $\mathcal{F}$. 
\end{lemm}
\begin{proof}
By the assumption, there exist Hilbert spaces $\mathcal{F}_{m}$ 
such that $\mathcal{F} = \varprojlim \mathcal{F}_{m}$. 
The operator $\varphi_{m}:\mathcal{H} \to \mathcal{F} \to \mathcal{F}_{m}$ 
induced by $\varphi$ is a compact operator. 
In particular, the vector $\varphi_{m}(w_{k})$ converges to 
$\varphi_{m}(w)$ in the Hilbert space $\mathcal{F}_{m}$. 
This means that $\varphi(w_{k})$ 
converges to $\varphi(w)$ in $\mathcal{F}$.
\end{proof}

\subsection{De Rham-Weil isomorphisms}\label{Sec2-6}

In this subsection, 
we observe the De Rham-Weil isomorphism 
from the $\dbar$-cohomology to the $\rm{\check{C}}$ech cohomology 
and give a refinement of \cite[Section 5]{Mat13} for our purpose. 
The contents in this subsection (including \cite[Section 5]{Mat13}) 
may be known for specialists, 
but we will explain them in detail for the reader's convenience.

Throughout this subsection, 
let $\omega$ be a K\"ahler form on a complex manifold $X$ 
and $(F,h)$ be a singular hermitian line bundle on $X$ 
such that $\sqrt{-1}\Theta_{h}(F) \geq -\omega$. 
Fix a locally finite 
open cover $\mathcal{U}:=\{B_{i} \}_{i \in I}$ of $X$ by 
sufficiently small Stein open sets $B_{i} \Subset X$.
We consider the set of $q$-cochains $C^{q}(\mathcal{U}, K_{X} \otimes F \otimes \I{h})$
with coefficients in $K_{X} \otimes F \otimes \I{h}$ calculated by $\mathcal{U}$ 
and the coboundary operator $\mu$ defined to be 
\begin{equation*}
\mu \{ \alpha_{i_{0}...i_{q}} \}_{i_{0}...i_{q}} := 
\{ \sum_{\ell = 0}^{q+1} (-1)^{\ell} 
\alpha_{  i_{0}...\hat{i_{\ell}}...i_{q+1} }
\ |_{B_{i_{0}...i_{q+1}} } \}_{i_{0}...i_{q+1}} 
\end{equation*}  
for every $q$-cochain $\{ \alpha_{i_{0}...i_{q}} \}_{i_{0}...i_{q}}$,   
where $B_{i_{0}...i_{q+1}} := B_{i_{0}}\cap B_{i_{1}}\cap \dots \cap B_{i_{q+1}}$.  
In this paper, we will omit the notation of the restriction, 
the subscript $\lq \lq i_{0}...i_{q}"$, and so on.  
The semi-norm $p_{K_{i_{0}...i_{q}}}(\bullet)$ is defined to be  
\begin{equation}\label{semi-defi}
p_{K_{i_{0}...i_{q}}}(\{\alpha_{i_{0}...i_{q}}\})^{2}:=
\int_{K_{i_{0}...i_{q}}} |\alpha_{i_{0}...i_{q}}|_{h, \omega}^{2} \, dV_{\omega}
\end{equation}
for a $q$-cochain $\{\alpha_{i_{0}...i_{q}}\}$ and 
a relatively compact set $K_{i_{0}...i_{q}} \Subset B_{i_{0}...i_{q}}$. 
Note that the semi-norm $p_{K_{i_{0}...i_{q}}}(\bullet)$ is independent of the choice of $\omega$ 
(see Lemma \ref{mul}). 
The set of $q$-cochains can be regarded as a topological vector space 
by the family of the semi-norms  $\{p_{K_{i_{0}...i_{q}}}(\bullet)\}_{K_{i_{0}...i_{q}} \Subset B_{i_{0}...i_{q}}}$. 
Then we have the following lemma.  

\begin{lemm}\label{Fre}
The set of $q$-cochains $C^{q}(\mathcal{U}, K_{X} \otimes F \otimes \I{h})$ 
and the set of $q$-cocycles 
$Z^{q}(\mathcal{U}, K_{X} \otimes F \otimes \I{h}):= {{\rm Ker}}\,\mu$ 
are Fr\'echet spaces with respect to the semi-norms 
$\{p_{K_{i_{0}...i_{q}}}(\bullet)\}_{K_{i_{0}...i_{q}} \Subset B_{i_{0}...i_{q}}}$. 
Moreover, if $X$ is holomorphically convex, 
then the set of $q$-coboundaries 
$B^{q}(\mathcal{U}, K_{X} \otimes F \otimes \I{h}):= {{\rm Im}}\,\mu$ is also 
a Fr\'echet space. 
\end{lemm}
\begin{proof}
We first remark that the topology of coherent ideal sheaves 
induced by the local sup-norms $\sup_{K_{i_{0}...i_{q}}}(\bullet)$ 
is equivalent to the topology 
induced by the local $L^2$-norms $p_{K_{i_{0}...i_{q}}}(\bullet)$ 
(for example see \cite[Theorem 2, Section D, Chapter I\hspace{-.1em}I]{GR65} 
or \cite[Lemma 5.2, Theorem 5.3, Lemma 5.7]{Mat13}). 
We can easily see that the metric induced by $p_{K_{i_{0}...i_{q}}}(\bullet)$ 
is complete, and thus $C^{q}(\mathcal{U}, K_{X} \otimes F \otimes \I{h})$ 
is a Fr\'echet space. 
It follows that 
$Z^{q}(\mathcal{U}, K_{X} \otimes F \otimes \I{h})= {{\rm Ker}}\,\mu$ 
is a closed subspace (in particular a Fr\'echet space) 
since the coboundary operator $\mu$ is continuous. 
For the latter conclusion, 
we consider the $\rm{\check{C}}$ech cohomology group
$$
\check{H}^{q}(\mathcal{U}, K_{X}\otimes F \otimes \I{h}):=
\dfrac{{\rm{Ker}}\, \mu}{{\rm{Im}}\, \mu} 
\text{ of } C^{q}(\mathcal{U}, K_{X}\otimes F \otimes \I{h}). 
$$
Since $X$ is a holomorphically convex, 
there exists a proper holomorphic map $\pi : X \to \Delta$ to a Stein space $\Delta$. 
Then, for a coherent sheaf $\mathcal{F}$ on $X$, 
the natural morphism 
\begin{equation}
\pi_{*}\, \colon \, H^{q}(X, \mathcal{F}) \to 
H^{0}(\Delta, R^{q}\pi_{*}\mathcal{F})
\end{equation}
is an isomorphism of topological vector spaces 
(for example, see \cite[Lemma II.1]{Pri71}). 
Hence we have the isomorphism between topological vector spaces 
\begin{equation*}
\pi_{*}\, \colon \, H^{q}(X, K_{X} \otimes F \otimes \I{h}) \to 
H^{0}(\Delta, R^{q}\pi_{*}(K_{X} \otimes F \otimes \I{h})). 
\end{equation*}
In particular, the $\rm{\check{C}}$ech cohomology group is 
a separated topological vector space when $X$ is holomorphically convex. 
Therefore $B^{q}(\mathcal{U}, K_{X} \otimes F \otimes \I{h})= {{\rm Im}}\,\mu$ 
must be closed. 
\end{proof}

Now we observe the $\dbar$-cohomology group of $L^2$-spaces defined on 
Zariski open sets equipped with suitable K\"ahler forms. 
Let $Z$ be a proper subvariety on $X$ and 
$\ome$ be a K\"ahler form on the Zariski open set 
$Y:=X \setminus Z$ with the following properties\,$:$
\begin{itemize}
\item[(B)] $\ome \geq \omega $ on $Y=X\setminus Z$. 
\item[(C)] For every point $p \in X$, 
there exists a bounded function $\Psi$ 
on an open neighborhood of $p$ such that $\ome =\deldel \Psi$. 
\end{itemize} 
The important point here is that $\ome$ locally admits 
a bounded potential function on $X$ (not $Y$). 
The above situation seems to be rather technical, 
but naturally appears in the proof of the main results. 
We define the local $L^{2}$-space of $F$-valued $(p, q)$-forms as follows\,: 
\begin{align}
& L_{(2, {\rm{loc}})}^{p, q}(F)_{h, \ome}:= \\ 
&\{u  \,|\,  u \text{ is an }F\text{-valued }(p, q)\text{-form with } 
\|u \|_{K, h, \ome}< \infty \text{ for every $K \Subset X$}. \}, \notag
\end{align}
where $\|u \|_{K, h, \omega}$ is the $L^2$-norm 
on a relatively compact set $K \Subset X$, that is, 
$$
\|u \|^2_{K, h, \ome}:=\int_{K \setminus Z} |u|^2 _{h, \ome}\, dV_{\ome}. 
$$

For the proof of Proposition \ref{DWiso}, 
we concretely construct  the De Rham-Weil isomorphism 
from the $\dbar$-cohomology to the $\rm{\check{C}}$ech cohomology
in Proposition \ref{DW-iso}.  
This construction plays an important role in the proof of the main results. 
To construct the De Rham-Weil isomorphism, 
we locally solve the $\dbar$-equation by Lemma \ref{LL}, 
which is obtained from the standard technique of 
the theory of Kodaira-Andreotti-Vesentini-H\"ormander 
(\cite{AV61}, \cite{AV65}, \cite{Hor65}, \cite{Kod53}).
The reader can check the proof of Lemma \ref{LL} 
in \cite[Lemma 5.4]{Mat13} with the same notation. 
We need to generalize Lemma \ref{LL} to Lemma \ref{loc-sol} 
for the proof of the main results.

\begin{lemm} \label{LL}
Under the same situation as above, 
we assume that $B$ is a sufficiently small Stein open set in $X$. 
Then, for an arbitrary 
$U \in {\rm{Ker}}\, \dbar \subset L^{n, q}_{(2)}(B\setminus Z, F)_{h, \ome}$,  
there exist $V \in L^{n, q-1}_{(2)}(B\setminus Z, F)_{h, \ome}$ 
and a positive constant $C$ $($depending only on $\Psi$, $q$$)$ 
such that 
\begin{align*}
 \dbar V = U \quad \text{ and }\quad 
\|V\|_{h, \ome}\leq C \|U\|_{h, \ome}. 
\end{align*}
\end{lemm}

\begin{prop}\label{DW-iso}
We consider the same situation as above. 
That is, we consider a singular hermitian line bundle $(F,h)$ on 
a K\"ahler manifold $(X,\omega)$ such that $\sqrt{-1}\Theta_{h}(F)\geq -\omega$ 
and a K\"ahler form $\ome$  on a Zariski open set $Y$ 
satisfying properties $(B)$, $(C)$. 
Then there exist continuous maps    
\begin{align*}
f: {\rm{Ker}}\, \dbar \text{ in } L^{n,q}_{(2, {\rm{loc}})}(F)_{h, \ome} 
\rightarrow 
{\rm{Ker}}\, \mu \text{ in } 
C^{q}(\mathcal{U}, K_{X} \otimes F \otimes \I{h}) \\
g: {\rm{Ker}}\, \mu \text{ in } 
C^{q}(\mathcal{U}, K_{X} \otimes F \otimes \I{h})
\rightarrow 
{\rm{Ker}}\, \dbar \text{ in } L^{n,q}_{(2, {\rm{loc}})}(F)_{h, \ome}
\end{align*}
satisfying the following properties\,$:$ 
\begin{itemize}
\item[$\bullet$] $f$ induces the isomorphism 
\begin{align*}
\overline{f} \colon \dfrac{{\rm{Ker}}\, \dbar}{{\rm{Im}}\, \dbar} 
\text{ of } L^{n,q}_{(2, {\rm{loc}})}(F)_{h, \ome}   
\xrightarrow{\quad \cong \quad }
\dfrac{{\rm{Ker}}\, \mu}{{\rm{Im}}\, \mu} 
\text{ of } C^{q}(\mathcal{U}, K_{X}\otimes F \otimes \I{h}).  
\end{align*}
\item[$\bullet$] $g$ induces the isomorphism 
\begin{align*}
\overline{g} \colon \dfrac{{\rm{Ker}}\, \mu}{{\rm{Im}}\, \mu} 
\text{ of } C^{q}(\mathcal{U}, K_{X}\otimes F \otimes \I{h})
\xrightarrow{\quad \cong \quad }
\dfrac{{\rm{Ker}}\, \dbar}{{\rm{Im}}\, \dbar} 
\text{ of } L^{n,q}_{(2, {\rm{loc}})}(F)_{h, \ome}.  
\end{align*}
\item[$\bullet$] $\overline{f}$ is the inverse map of $\overline{g}$. 
\end{itemize}
\end{prop}
\begin{proof}
The construction is essentially the same as the standard De Rham-Weil isomorphism. 
We briefly review only the construction of $f$ and $g$. 
See \cite[Proposition 5.5]{Mat13} and \cite[Lemma 3.20]{Fuj13} for more details.

We first define $f(U)$ 
for $U \in {\rm{Ker}}\, \dbar \subset L^{n,q}_{(2, {\rm{loc}})}(F)_{h, \ome}$ 
by using the local solution of the $\dbar$-equation with minimum $L^2$-norm. 
We consider the $\dbar$-equation 
$\dbar \beta_{i_{0}} =U |_{B_{i_{0}}\setminus Z}$ on $B_{i_{0}} \setminus Z$. 
Then we can take the solution $\beta_{i_{0}}$ whose $L^2$-norm 
is minimum among all solutions. 
Further, we can see that
$\|\beta_{i_{0}}\|_{h, \ome} \leq C \|U\|_{B_{i_{0}}, h, \ome} \leq 
C \|U\|_{h, \ome}$ 
for some constant $C$ (independent of $U$) by Lemma \ref{LL}. 
Next we take the solution of the $\dbar$-equation 
$\dbar \beta_{i_{0}i_{1}} =\beta_{i_{1}}-\beta_{i_{0}}$ 
on $B_{i_{0}i_{1}} \setminus Z$ with minimum $L^2$-norm. 
By Lemma \ref{LL}, we have 
$$
\|\beta_{i_{0}i_{1}}\|_{h, \ome} \leq 
D \|\beta_{i_{1}}-\beta_{i_{0}}\|_{B_{i_{0}i_{1}}, h, \ome} \leq 
D (\|\beta_{i_{1}}\|_{h, \ome}+ \|\beta_{i_{0}}\|_{h, \ome}) \leq 
2CD \|U\|_{h, \ome}
$$
for some constant $D$. 
By repeating this process, 
we can obtain the $F$-valued $(n,q-k-1)$-forms 
$\beta_{i_{0}\dots i_{k}}$ 
on $B_{i_{0}\dots i_{k}} \setminus Z$ 
satisfying 
\[
  (*) \left\{ \quad
  \begin{array}{ll}
\vspace{0.2cm}
\dbar \{ \beta_{i_{0}} \} &=\{U |_{B_{i_{0}}\setminus Z}\},  \\
\dbar \{ \beta_{i_{0}i_{1}} \}&=\mu  \{\beta_{i_{0}}\},  \\
\dbar \{ \beta_{i_{0}i_{1}i_{2}} \}&
=\mu   \{\beta_{i_{0}i_{1}}\},  \\
 & \vdots   \\
\dbar \{ \beta_{i_{0}\dots i_{q-1}} \}
&=\mu   \{\beta_{i_{0}\dots i_{q-2}}\}.  
  \end{array} \right.
\]
Then $\mu \{ \beta_{i_{0}\dots i_{q-1}} \}=:\{\beta_{i_{0}\dots i_{q}}\}$ is a $q$-cocycle 
of $\dbar$-closed $F$-valued $(n,0)$-forms on $B_{i_{0}...i_{q}} \setminus Z$ 
such that $\|\beta_{i_{0}\dots i_{q}} \|_{h, \ome} < \infty$. 
We remark that 
$\|\beta_{i_{0}\dots i_{q}} \|_{h, \ome}=
\|\beta_{i_{0}\dots i_{q}} \|_{h, \omega}$ by Lemma \ref{mul}. 
We locally regard $\beta_{i_{0}\dots i_{q}}$
as a holomorphic function.
Then, by the Riemann extension theorem, 
$\beta_{i_{0}\dots i_{q}}$ can be extend 
from $B_{i_{0}...i_{q}} \setminus Z$ to $B_{i_{0}...i_{q}}$. 
Therefore $\mu \{ \beta_{i_{0}\dots i_{q-1}} \}
=\{\beta_{i_{0}\dots i_{q}}\}$ determines  
a $q$-cocycle in $C^{q}(\mathcal{U}, K_{X}\otimes F \otimes \I{h})$. 
We define $f$ by $f(U):=\mu \{ \beta_{i_{0}\dots i_{q-1}} \}$. 

\begin{rem}\label{DW-rem}
By the construction, 
we can obtain the $L^2$-estimate 
$\|\beta_{i_{0}\dots i_{k}}\|_{h,\ome} \leq C \|U\|_{h,\ome} $ 
for some constant $C$. 
Here $C$ essentially depends on a constant 
that appears when we solve the $\dbar$-equation by the $L^2$-method, 
such as Lemma \ref{LL}. 
We remark that the constant $C$ does not depend on $U$. 
\end{rem}

In order to define $g$, we fix a partition of unity $\{\rho_{i} \}_{i \in I}$ of 
$\mathcal{U}=\{B_{i}\}_{i \in I}$. 
For a $q$-cocycle $\alpha=\{\alpha_{i_{0}...i_{q}}\}\in 
C^{q}(\mathcal{U}, K_{X}\otimes F \otimes \I{h})$, 
we define $g(\alpha)$ by 
\begin{align*}
& g(\alpha)\\
:=&\dbar \Bigg(  \sum_{k_{q}\in I}\rho_{k_{q}}
\dbar \bigg( \sum_{k_{q-1}\in I}\rho_{k_{q-1}} \cdots
\dbar \Big( \sum_{k_{3}\in I}\rho_{k_{3}}
\dbar \big( \sum_{k_{2}\in I}\rho_{k_{2}}
\dbar (\sum_{k_{1}\in I}\rho_{k_{1}} \alpha_{k_{1}...k_{q}i_{0}})
\big)\Big)\bigg) \Bigg)\\
=&
\sum_{k_{q}\in I} \dbar \rho_{k_{q}}\wedge
\sum_{k_{q-1}\in I} \dbar \rho_{k_{q-1}} \wedge \cdots
\sum_{k_{3}\in I} \dbar \rho_{k_{3}} \wedge
\sum_{k_{2}\in I} \dbar \rho_{k_{2}} \wedge
\dbar (\sum_{k_{1}\in I}\rho_{k_{1}} \alpha_{k_{1}...k_{q}i_{0}}). 
\end{align*}
This determines the $\dbar$-closed $F$-valued $(n,q)$-form 
with locally bounded $L^2$-norm. 
\end{proof}

From Proposition \ref{DW-iso}, we obtain the following proposition.  

\begin{prop}\label{closed}
Under the same situation as in Proposition \ref{DW-iso}, 
we assume that $X$ is holomorphically convex. 
Then ${{\rm{Im}}\, \dbar}$ is a closed subspace in 
$L^{n,q}_{(2, {\rm{loc}})}(F)_{h, \ome}$. 
In particular the $\dbar$-cohomology group  
$\rm{Ker}\, \dbar/\rm{Im}\, \dbar$ of  $L^{n,q}_{(2, {\rm{loc}})}(F)_{h, \ome}$ 
is a Fr\'echet space. 
\end{prop}
\begin{proof}
Since $X$ is holomorphically convex, 
the $\rm{\check{C}}$ech cohomology group is a separated topological vector space 
(see Lemma \ref{Fre}). 
Therefore ${{\rm{Im}}\, \dbar}$ 
is a closed subspace by Proposition \ref{DW-iso}. 
\end{proof}

We close this subsection with the following proposition\,:

\begin{prop}\label{DWiso}
Under the same situation as in Proposition \ref{DW-iso}, 
we assume that $X$ is holomorphically convex. 
Then the following composite map is a compact operator. 
\[\xymatrix{
 {\rm{Ker}}\, \dbar \text{ in } 
L^{n,q}_{(2)}(Y,F)_{h, \ome} \ar[r] & 
{\rm{Ker}}\, \dbar \text{ in } 
L^{n,q}_{(2, {\rm{loc}})}(F)_{{h}, \ome}
\ar[r] & 
\dfrac{{\rm{Ker}}\, \dbar}{{\rm{Im}}\, \dbar} 
\text{ of } L^{n,q}_{(2, {\rm{loc}})}(F)_{h, \ome}. 
}\]
\end{prop}
\begin{rem}
By Proposition \ref{closed}, 
the $\dbar$-cohomology group of 
$L^{n,q}_{(2, {\rm{loc}})}(F)_{h, \ome}$ is a 
Fr\'echet space. 
The natural quotient map 
$${\rm{Ker}}\, \dbar \text{ in } 
L^{n,q}_{(2, {\rm{loc}})}(F)_{{h}, \ome} \to 
\dfrac{{\rm{Ker}}\, \dbar}{{\rm{Im}}\, \dbar} 
\text{ of } L^{n,q}_{(2, {\rm{loc}})}(F)_{h, \ome}$$ 
is not a compact operator, and thus we need to consider 
the map defined on ${\rm{Ker}}\, \dbar \subset L^{n,q}_{(2)}(Y,F)_{h, \ome}$. 
\end{rem}

In the proof of Proposition \ref{DWiso}, 
we use the construction of $f$ and $g$ 
in Proposition \ref{DW-iso}.  
For the reader's convenience, we first give a proof for the case $q=1$.

\begin{proof}[Proof of Proposition \ref{DWiso} for the case $q=1$]
We take a bounded sequence $\{U_{\ell}\}_{\ell=1}^{\infty} $ 
in ${\rm{Ker}}\, \dbar \subset  L^{n,1}_{(2)}(Y,F)_{h, \ome}$, 
that is,  $F$-valued $(n,1)$-forms 
$\{U_{\ell}\}_{\ell=1}^{\infty} \subset 
{\rm{Ker}}\, \dbar \subset  L^{n,1}_{(2)}(Y,F)_{h, \ome}$ 
such that $\|U_{\ell}\|_{h,\ome} \leq C_{1}$  for some constant 
$C_{1}$ (independent of $\ell$). 
For the restriction $U_{\ell,i}:=U_{\ell}|_{B_{i} \setminus Z}$, 
by solving the $\dbar$-equation 
$\dbar \beta_{\ell,i}= U_{\ell,i}$
on $B_{i} \setminus Z$,  
we obtain $\beta_{\ell,i}$ and a constant $C$ independent of $\ell$ 
with the following properties\,: 
$$ 
\dbar \beta_{\ell,i} = U_{\ell,i}\quad \text{on}\ B_{i}\setminus Z
\quad \text{and} \quad 
\|\beta_{\ell,i}\|_{h, \ome} \leq 
C \|U_{\ell,i}\|_{B_{i}, h, \ome}\leq C \|U_{\ell}\|_{h, \ome}. 
$$
Here we used the assumption that $\ome$ locally admits  
a bounded potential function $\Psi$ (see Lemma \ref{LL}). 
In particular, 
the $F$-valued $(n,0)$-form $\beta_{\ell,j} - \beta_{\ell,i}$ 
is $\dbar$-closed on $B_{ij} \setminus Z$ 
and it is (uniformly) $L^2$-bounded. 
Therefore it can be extended to 
the $\dbar$-closed  $F$-valued $(n,0)$-form on $B_{i}\cap B_{j}$
by the Riemann extension theorem.

From now on,  
we construct the $F$-valued $(n,1)$-form $V_{\ell}$
such that $V_{\ell}$ determines the same cohomology class as $U_{\ell}$ 
and $V_{\ell}$ converges to some $F$-valued $(n,1)$-form. 
(This completes the proof.)
By the construction, 
the sup-norm $\sup_{K}|\beta_{\ell,j} -\beta_{\ell,i}|$ is 
uniformly bounded for every $K \Subset B_{ij}$. 
(Recall the local sup-norm of holomorphic functions 
can be bounded by the $L^{2}$-norm). 
Therefore, 
by Montel's theorem, there exists 
a subsequence of $\{\beta_{\ell,j} - \beta_{\ell,i}\}_{\ell=1}^{\infty}$ 
such that it uniformly converges to some $\alpha_{ ij}$ as $\ell \to \infty$ 
on every relatively compact set in $B_{ij}$. 
We use the same notation for this subsequence. 
This subsequence also converges to $\alpha_{ij}$ 
with respect to the local $L^2$-norms (for example see \cite[Lemma 5.2]{Mat13}), 
that is, $p_{K_{ij}}(\beta_{\ell,j} - \beta_{\ell,i}-\alpha_{ ij}) \to 0$.

For a fixed partition of unity $\{\rho_{i}\}_{i \in I}$ 
of $\mathcal{U}$, 
we define the $F$-valued $(n,1)$-form $V_{\ell,i}$ by 
$$
V_{\ell,i}:=\dbar \big( \sum_{k\in I} \rho_{k} 
(\beta_{\ell, i}-\beta_{\ell, k}) \big) = \dbar (\beta_{\ell, i}-\sum_{k\in I} \rho_{k} \beta_{\ell, k}) 
=U_{\ell}-\dbar (\sum_{k\in I} \rho_{k} \beta_{\ell, k})
\text{ on }  B_{i}\setminus Z. 
$$ 
Since $\sum_{k\in I} \rho_{k} \beta_{\ell, k}$ is independent of $i$, 
the family $\{V_{\ell,i}\}_{i \in I}$ determines 
the $F$-valued $(n,1)$-form $V_{\ell}$ globally defined on $Y$. 
Further the $F$-valued $(n,1)$-form $V_{\ell}$  
determines the same cohomology class as $U_{\ell}$. 
It is sufficient for the proof to show that 
$V_{\ell}$ converges to  
$\sum_{k\in I} \dbar (\rho_{k} \alpha_{ki})=
\sum_{k\in I} \dbar \rho_{k} \wedge \alpha_{ki}$ in 
$L^{n,1}_{(2, {\rm{loc}})}(F)_{h, \ome}$. 
For a given $K \Subset X$, 
the cardinality of $I_{K}$ defined by 
$$
I_{K}:=\{i \in I \, | \, B_{i} \cap K \not=\emptyset \}
$$ 
is finite. 
Hence we obtain 
\begin{align*}
\| V_{\ell} - \sum_{k\in I} \dbar (\rho_{k}  \alpha_{ki}) \|_{K, h, \ome}
&\leq 
\sum_{i\in I_{K} }\| \sum_{k\in I}  \dbar \rho_{k} \wedge
(\beta_{\ell, i}-\beta_{\ell, k}- \alpha_{ki})\|_{B_{i}, h, \ome}\\
&\leq C_{2}\| 
(\beta_{\ell, i}-\beta_{\ell, k}- \alpha_{ki})\|_{B_{i} \cap {\rm{Supp}}\,\rho_{k}, h, \ome} 
\end{align*}
for some constant $C_{2}>0$. 
Here we used $|\dbar \rho_{k}|_{\ome} \leq |\dbar \rho_{k}|_{\omega}$
(see Lemma \ref{mul} and property (B)). 
By $K_{i} \cap {\rm{Supp}}\rho_{k} \Subset B_{ik}$, 
the right hand side converges to zero as $\ell$ tends to $\infty$. 
This completes the proof. 
\end{proof}

\begin{proof}[Proof of Proposition \ref{DWiso} for the general case]
We take a bounded sequence $\{U_{\ell}\}_{\ell=1}^{\infty} $ in $ 
{\rm{Ker}}\, \dbar \subset  L^{n,q}_{(2)}(Y,F)_{h, \ome}$. 
Then, by the construction of $f$, 
we can obtain the $F$-valued $(n,q-k-1)$-forms 
$\beta_{\ell, i_{0}\dots i_{k}}$ 
on $B_{i_{0}\dots i_{k}} \setminus Z$ 
satisfying equality $(*)$. 
Further 
we obtain 
$$
f(U_{\ell})=\mu \{ \beta_{\ell, i_{0}\dots i_{q-1}} \}
=:\{ \beta_{\ell, i_{0}\dots i_{q}} \}
\quad \text{ and }\quad 
\|\beta_{\ell, i_{0}\dots i_{q}}\|_{B_{i_{0}\dots i_{q}}, h,\ome} \leq 
C \|U_{\ell}\|_{h,\ome}.  
$$
Here $C$ is a positive constant independent of $\ell$ (see Remark \ref{DW-rem}).
Note that $\beta_{\ell, i_{0}\dots i_{q}}$ can be regarded as a holomorphic 
function on $B_{i_{0}\dots i_{q}}$ 
since it is a $\dbar$-closed $(n,0)$-form. 
The local sup-norm $\sup_{K} |\beta_{\ell, i_{0}\dots i_{q}}|$ 
is uniformly bounded for every relatively compact set 
$K \Subset B_{i_{0}\dots i_{q}}$.  
Therefore, by Montel's theorem, we can take a subsequence of 
$\{f(U_{\ell})\}_{\ell=1}^{\infty}$ converging some $q$-cocycle $\alpha$ 
with respect to the local sup-norms. 
This subsequence converges to $\alpha$ also with respect to 
the local $L^2$-norms. 
We use the same notation $\{f(U_{\ell})\}_{\ell=1}^{\infty}$ for 
this subsequence.

When we fix a partition of unity $\{\rho_{i} \}_{i \in I}$ of 
$\mathcal{U}$, 
we can define the map $g$ in Proposition \ref{DW-iso}. 
It follows that $g(f(U_{\ell}))$ converges to $g(\alpha)$ 
in $L^{n,q}_{(2, {\rm{loc}})}(F)_{h, \ome}$ 
since the map $g$ is continuous. 
Further 
we can see that $g(f(U_{\ell}))$  
determines the same cohomology class as 
$U_{\ell}$ by Proposition \ref{DW-iso}. 
This completes the proof. 
\end{proof}

\section{Proof of the main results}\label{Sec3}
\subsection{Proof of Theorem \ref{main}}\label{Sec3-2}
In this subsection, we prove Theorem \ref{main}. 
The proof can be divided into five steps. 
The proof of Theorem \ref{main2} will be obtained from  
a slight revision of Step \ref{S5} (see subsection \ref{Sec3-2}). 
In the following proof, 
we often write (possibly different) positive constants as $C$ 
and use the same notation for suitably chosen subsequences.

\begin{step}[Reduction of the proof to the local problem]\label{S1}

In this step, we fix the notation used in this subsection 
and reduce the proof to a local problem on $\Delta$. 
At the end of this step, the rough strategy of the proof 
will be given, 
which helps us to understand the complicated and technical arguments.

The problem is local on $\Delta$, 
and thus we may assume that $\pi\,\colon\,X \to \Delta$ 
is a surjective proper holomorphic map from 
a  K\"ahler manifold $X$ to 
a Stein subvariety $\Delta $ in $ \mathbb{C}^{N}$. 
Then the manifold $X$ is holomorphically convex 
since $X$ admits a proper holomorphic map to a Stein space. 
In particular, for a coherent sheaf $\mathcal{F}$ on $X$, 
the natural morphism 
\begin{equation}
\pi_{*}\, \colon \, H^{q}(X, \mathcal{F}) \to 
H^{0}(\Delta, R^{q}\pi_{*}\mathcal{F})
\end{equation}
is an isomorphism of topological vector spaces 
(for example, see \cite[Lemma II.1]{Pri71}). 
Therefore it is sufficient to show that 
the multiplication map 
$$
H^{q}(X, K_{X}\otimes F \otimes \I{h}) \xrightarrow{\otimes s} 
H^{q}(X, K_{X}\otimes F^{m+1} \otimes \I{h^{m+1}})
$$
is injective. 
In \cite{Mat13}, 
we have already proved that 
the above multiplication map is injective when $X$ is compact. 
One of the difficulties of the proof is 
to deal with the non-compact manifold $X$.

By replacing $\Delta$ with smaller one (if necessary), 
we may assume that $X$ is a relatively compact set 
in the initial ambient space. 
In particular, the point-wise norm $|s|_{h^{m}}$ can be assumed to be  
bounded on $X$ by the assumption. 
Note that $X$ admits  
a complete K\"ahler form 
since $X$ is a weakly pseudoconvex K\"ahler manifold. 
For a fixed complete K\"ahler form  $\omega$ on $X$, 
by applying Theorem \ref{equi} and Remark \ref{equi-rem} for $\gamma=0$,  
we can take a family of singular metrics 
$\{h_{\e} \}_{1\gg \e>0}$ on $F$ with 
the following properties\,: 
\begin{itemize}
\item[(a)] $h_{\e}$ is smooth on $X \setminus Z_{\e}$ 
for some proper subvariety $Z_{\e}$. 
\item[(b)]$h_{\e''} \leq h_{\e'} \leq h$ holds on $X$ 
for any $0< \e' < \e'' $.
\item[(c)]$\I{h}= \I{h_{\e}}$ on $X$.
\item[(d)]$\sqrt{-1} \Theta_{h_{\e}}(F) \geq -\e \omega$ on $X$. 
\end{itemize}

\begin{rem}\label{m0}
In the case $m>0$, 
the set $\{x \in X\,|\,\nu(h,x)>0\}$ 
is contained in the zero set $s^{-1}(0)$ of $s$ 
by $\sup |s|_{h^{m}} < \infty$, 
where $\nu(h,x)$ is the Lelong number of the weight of $h$ at $x$. 
Then we can assume that $Z_{\e}$ is independent of $\e$. 
However, in the case $m=0$, 
the subvariety $Z_{\e}$ may essentially depend on $\e$, 
which is different from \cite{Mat13}. 
For this reason, we need to consider a complete K\"ahler form $\omega_{\e, \delta}$ on $Y_{\e}:=X \setminus Z_{\e}$ 
such that $\omega_{\e, \delta}$ converges to 
$\omega$ as $\delta$ goes to zero. 
\end{rem}

To use the theory of harmonic integrals 
on the Zariski open set  $Y_{\e}$, 
we first take a complete K\"ahler form  $\omega_{\e}$ on $Y_{\e}$ 
with the following properties\,$:$ 
\begin{itemize}
\item[$\bullet$] $\omega_{\e}$ is a complete K\"ahler form on 
$Y_{\e}$.
\item[$\bullet$] $\omega_{\e} \geq \omega $ on $Y_{\e}$. 
\item[$\bullet$] $\omega_{\e}=\deldel \Psi_{\e}$ 
for some bounded function $\Psi_{\e}$ 
on a neighborhood of every $p \in X$.    
\end{itemize} 
See \cite[3.11]{Fuj13} for the construction of $\omega_{\e}$. 
The key point here is the third property on the bounded potential function, 
which enables us to construct the De Rham-Weil isomorphism 
from  the $\dbar$-cohomology group on $Y_{\e}$ to 
the $\rm{\check{C}}$ech cohomology group on $X$ (see Proposition \ref{DW-iso}).  
We define the K\"ahler form $\omega_{\e, \delta}$ 
on $Y_{\e}$ by 
$$
\omega_{\e, \delta}:=\omega + \delta \omega_{\e} 
\text{ for } \e \text{ and } \delta  \text{ with } 0<\delta \ll \e. 
$$
It is easy to see the following properties\,$:$ 
\begin{itemize}
\item[(A)] $\omega_{\e, \delta}$ is a complete K\"ahler form on $Y_{\e}=X\setminus Z_{\e}$ for every $\delta > 0$.
\item[(B)] $\omega_{\e, \delta} \geq \omega $ on $Y_{\e}$ 
for every  $\delta>0$. 
\item[(C)] For every point $p \in X$, there exists a bounded function 
$\Psi_{\e,\delta}=\Psi + \delta \Psi_{\e}$
on an open neighborhood of $p$ in $X$ such that 
$\deldel \Psi_{\e,\delta}=\omega_{\e,\delta}$ and $\lim_{\delta \to 0}\Psi_{\e,\delta}=\Psi$. Here $\Psi$ is a local potential function of $\omega$. 
\end{itemize} 

\begin{rem}\label{countable}
Strictly speaking, by Theorem \ref{equi} (\cite[Theorem 2.3]{DPS01}), 
we obtain a countable family $\{h_{\e_{k}} \}_{k=1}^{\infty}$ 
of singular metrics satisfying the above properties and $\e_{k} \to 0$. 
In our proof, 
we actually consider only countable sequences $\{\e_{k}\}_{k=1}^{\infty}$ 
and $\{\delta_{\ell}\}_{\ell=1}^{\infty}$ conversing to zero 
since we need to use Cantor's diagonal argument, 
but we often use the notations $\e$ and $\delta$ for simplicity. 
\end{rem}

We define the function $\Phi$ on $X$ by 
\begin{equation}
\Phi:=\pi^{*}i^{*}(|z_{1}|^{2}+|z_{2}|^{2}+\dots+|z_{N}|^{2}), 
\end{equation}
where $i \colon \Delta \to \mathbb{C}^{N}$ is a local embedding 
of the Stein subvariety $\Delta$ 
and $(z_{1}, z_{2}, \dots, z_{N})$ is a coordinate 
of $\mathbb{C}^{N}$. 
By the construction, 
the function $\Phi$ is a psh function on $X$. 
Since $\pi$ is a proper morphism, 
the function $\Phi$ is an exhaustive function on $X$ 
(that is, the level set $X_{c}:=\{x \in X\, |\, \Phi(x)<c \}$
is relatively compact in $ X$ for every $c$ with $c < \sup_{X}\Phi$). 
Moreover, we may assume that 
$$\sup_{X} \Phi < \infty \quad \text{ and }\quad  \sup_{X} |d \Phi|_{\omega_{\e, \delta}}<C$$
by replacing $\Delta$ with smaller one.   
Indeed, we can assume that 
$\Phi$ is defined on a neighborhood of $\partial X$ 
in the initial ambient space 
by taking smaller $\Delta$. 
This implies that $\sup_{X} \Phi < \infty$ and 
$\sup_{X} |d \Phi|_{\overline{\omega}} <\infty$
for some positive $(1,1)$-form $\overline{\omega}$ defined on 
the neighborhood of $\partial X$. 
We may assume that $\omega_{\e, \delta} \geq 
\omega \geq \overline{\omega}$
since $\omega$ is a complete form on $X$. 
By Lemma \ref{mul}, we have the inequality 
$\sup_{X} |d \Phi|_{\omega_{\e, \delta}} \leq 
\sup_{X} |d \Phi|_{\overline{\omega}} < \infty$. 
In particular, it was shown that the function $\Phi$ and the complete K\"ahler form 
$\omega_{\e, \delta}$ satisfy the assumptions 
in Proposition \ref{Nak}.

Let $A$ be a cohomology class in $H^{q}(X, K_{X}\otimes F \otimes \I{h})$ 
such that  $sA =0 \in H^{q}(X, K_{X}\otimes F^{m+1} \otimes \I{h^{m+1}})$. 
Our goal is to prove that $A$ 
is actually the zero cohomology class.

We briefly explain the strategy of the proof 
with the above notations. 
In Step \ref{S2}, by suitably choosing an increasing convex function 
$\chi\, \colon\, \mathbb{R} \to \mathbb{R}$, 
we represent $A$ by 
harmonic $L^2$-forms $u_{\e, \delta}$ with respect to 
$\omega_{\e, \delta}$ and 
the new metric $H_{\e}$ 
defined by $H_{\e}:=h_{\e}e^{-\chi\circ \Phi}$.  
In Step \ref{S3}, we consider the level set 
$X_{c}:= \{x\in X\, |\, \Phi(x) < c \}$, 
and show that if the $L^2$-norm 
$\|su_{\e, \delta}\|_{X_{c}, H_{\e}h_{\e}^m, \omega_{\e, \delta}}$ 
on $X_{c}$ converges to zero for almost all $c$, 
then $A$ is zero. 
To prove this convergence, 
in Step \ref{S4}, 
we construct a solution $v_{\e, \delta}$ of 
the $\dbar$-equation $\dbar v_{\e, \delta} = su_{\e, \delta}$ 
such that   
the $L^{2}$-norm $\| v_{\e, \delta} \|_{X_{c}, H_{\e},\omega_{\e, \delta}}$ on $X_{c}$  
is uniformly bounded. 
In Step \ref{S5},  we show that  
\begin{align*}
\| su_{\e, \delta} \|_{X_{c}, H_{\e}h_{\e}^m,\omega_{\e, \delta}} ^{2} &= 
\lla su_{\e, \delta}, \dbar v_{\e, \delta} \rra_{X_{c}, H_{\e}h_{\e}^m,\omega_{\e, \delta}} \\
&=\lla \dbar^{*} su_{\e, \delta}, v_{\e, \delta} \rra_{X_{c}, H_{\e}h_{\e}^m,\omega_{\e, \delta}} 
+\la (d\Phi)^{*} su_{\e, \delta}, v_{\e, \delta} \ra_{\partial X_{c}, H_{\e}h_{\e}^m,\omega_{\e, \delta}}
\to 0 
\end{align*}
by using the twisted Bochner-Kodaira-Nakano identity. 
\end{step}

\begin{step}[$L^2$-spaces and representations by harmonic forms]\label{S2}
In this step, 
we construct harmonic $L^2$-forms $u_{\e, \delta}$ representing 
the cohomology class $A \in H^{q}(X, K_{X}\otimes F \otimes \I{h})$, 
and prove Proposition \ref{finish}, 
which says that if $u_{\e, \delta}$ converges to zero 
in a suitable sense, 
then $A$ is zero (that is, the proof is completed).

We first consider the standard De Rham-Weil isomorphism\,:
\begin{equation}\label{DW}
H^{q}(X, K_{X}\otimes F \otimes \I{h})
\cong 
\frac{\, {\rm{Ker}}\,\dbar: L^{n,q}_{(2, {\rm{loc}})}(F)_{h,\omega}
\to L^{n,q+1}_{(2, {\rm{loc}})}(F)_{h,\omega}}
{{\rm{Im}}\,\dbar: L^{n,q-1}_{(2, {\rm{loc}})}(F)_{h,\omega}
\to L^{n,q}_{(2, {\rm{loc}})}(F)_{h,\omega}}, 
\end{equation}
where  $L^{n,\bullet}_{(2, {\rm{loc}})}(F)_{h,\omega}$ is 
the set of $F$-valued $(n,\bullet)$-forms $f$ on $X$ 
with locally bounded $L^{2}$-norm 
(that is, the $L^{2}$-norm $\|f\|_{K,h,\omega}$ on $K$ 
with respect to $h$ and $\omega$
is finite  
for every relatively compact set $K\Subset X$).  
By the above isomorphism, 
the cohomology class $A$ can be represented by 
a $\dbar$-closed $F$-valued $(n,q)$-form $u$ on $X$ 
with locally bounded  $L^{2}$-norm. 
Our goal is to prove that 
$u \in {\rm{Im}}\,\dbar \subset L^{n,q}_{(2, {\rm{loc}})}(F)_{h,\omega}$.

We want to represent the cohomology class $A$ by harmonic forms 
in $L^2$-spaces, but unfortunately, 
$u$ may not be globally $L^{2}$-integrable on $X$. 
For this reason, we construct a new metric  on $F$ 
that makes $u$ globally $L^{2}$-integrable. 
Since $u$ is locally $L^{2}$-integrable and 
$\Phi$ is exhaustive, 
there exists an increasing convex function 
$\chi\, \colon\, \mathbb{R} \to \mathbb{R}$ 
such that the $L^{2}$-norm 
$\|u\|_{h e^{-\chi(\Phi)}, \omega}$ on $X$ is finite.  
For simplicity we put
$$
H:=h e^{-\chi(\Phi)}, \quad  \quad 
H_{\e}:=h_{\e} e^{-\chi(\Phi)}, \quad \text{and} \quad 
\|u\|_{\e, \delta}:=\|u\|_{H_{\e}, \omega_{\e, \delta}}. 
$$
Then we obtain the following inequality\,:  
\begin{align}\label{ineq-1}
\|u\|_{\e, \delta} \leq \|u\|_{H, \omega_{\e, \delta}} \leq \|u\|_{H, \omega} <\infty.  
\end{align}
Strictly speaking, the left hand side should be 
$\|u|_{Y_{\e}}\|_{\e, \delta}$, 
but we often omit the symbol of restriction.  
The first inequality follows from property (b) of $h_{\e}$, 
and the second inequality follows from 
Lemma \ref{mul} and property (B) of $\omega_{\e, \delta}$. 
Here we used a special characteristic of the canonical bundle $K_{X}$ 
since the second inequality holds only for $(n,q)$-forms. 
The norm $\|u\|_{\e, \delta}$ is uniformly bounded 
since the right hand side is independent of $\e$, $\delta$. 
These inequalities play an important role in the proof.

We consider 
the $L^{2}$-space 
$$
L^{n,q}_{(2)}(F)_{\e, \delta}:=
L^{n,q}_{(2)}(Y_{\e},F)_{H_{\e},\omega_{\e, \delta}}
$$ 
on $Y_{\e}$ with respect to 
$H_{\e}$ and $\omega_{\e, \delta}$ (not $H$ and $\omega$). 
In general, we have the following orthogonal decomposition\,: 
$$
L^{n,q}_{(2)}(F)_{\e, \delta}= 
\overline{\rm{Im}\, \dbar}\,  \oplus 
\mathcal{H}^{n,q}_{\e, \delta}(F)\, \oplus 
\overline{\rm{Im}\, \dbar^{*}_{\e, \delta}}, 
$$
where $\overline{\bullet}$ denotes the closure of $\bullet$ with respect to the $L^{2}$-topology and 
$\mathcal{H}^{n,q}_{\e, \delta}(F)$ denotes the set of 
harmonic $F$-valued $(n,q)$-forms on $Y_{\e}$, namely    
$$
\mathcal{H}^{n,q}_{\e, \delta}(F):= 
\{ v \in L^{n,q}_{(2)}(F)_{\e, \delta} \, | \, 
\dbar v=\dbar^{*}_{\e, \delta}v=0 \}. 
$$
We remark that (the maximal extension of) 
the formal adjoint $\dbar^{*}_{\e, \delta}$ 
agrees with the Hilbert space adjoint
since $\omega_{\e, \delta}$ is complete for 
$\delta >0$ (see Lemma \ref{density}). 
Strictly speaking, 
$\dbar$ also depends on $H_{\e}$ and $\omega_{\e, \delta}$ 
since the domain and range of the closed operator $\dbar$ depend on them,  
but we abbreviate $\dbar_{\e, \delta}$ to $\dbar$.
  
The $F$-valued $(n,q)$-form $u$ belongs to $L^{n,q}_{(2)}(F)_{\e,\delta}$ 
by inequality $(\ref{ineq-1})$. 
By the above orthogonal decomposition, 
the $F$-valued $(n,q)$-form $u$ can be decomposed as follows\,:  
\begin{align}\label{ortho}
u=w_{\e, \delta}+u_{\e, \delta}\quad 
\text{for some } 
w_{\e, \delta} \in \overline{\rm{Im}\, \dbar}\text{ and } 
u_{\e, \delta} \in \mathcal{H}^{n,q}_{\e, \delta}(F) 
\text{ in } L^{n,q}_{(2)}(F)_{\e,\delta}. 
\end{align}
Note that the orthogonal projection of $u$ to 
$\overline{\rm{Im}\, \dbar^{*}_{\e, \delta}}$ 
is zero by the following lemma.

\begin{lemm}\label{tiny}
If $u$ belongs to ${\rm{Ker}}\, \dbar$, then 
the orthogonal projection of $u$ 
to $\overline{\rm{Im}\, \dbar^{*}_{\e, \delta}}$  is zero. 
\end{lemm}
\begin{proof}
For an arbitrary element $\lim_{k \to \infty}
\dbar^{*}_{\e, \delta} c_{k} \in \overline{\rm{Im}\, \dbar^{*}_{\e, \delta}}$, 
we have
\begin{align*}
\lla u, \lim_{k \to \infty}\dbar^{*}_{\e, \delta} c_{k} \rra_{\e,\delta} 
=\lim_{k \to \infty} \lla u, \dbar^{*}_{\e, \delta} c_{k} \rra_{\e,\delta} 
=\lim_{k \to \infty} \lla \dbar u,  c_{k} \rra_{\e,\delta} =0.
\end{align*}
This leads to the conclusion. 
\end{proof}

From now on, we take a suitable limit of $u_{\e, \delta}$. 
We need to carefully choose the $L^2$-space,
since the $L^2$-space $L^{n,q}_{(2)}(F)_{\e,\delta}$ 
depends on $\e, \delta$ 
although we have property (c). 
We remark that $\{ \e \}_{\e>0}$ and $\{ \delta \}_{\delta>0}$ 
denote countable sequences converging to zero (see Remark \ref{countable}). 
Let $\{ \delta_{0} \}_{\delta_{0}>0}$ be another countable
sequence converging to zero.

\begin{prop}\label{limit}
There exist a subsequence $\{\delta_{\nu}\}_{\nu=1}^{\infty}$ of 
$\{\delta\}_{\delta>0}$ and 
$\alpha_{\e} \in L^{n,q}_{(2)}(F)_{H_{\e}, \omega}$ with 
the following properties\,$:$
\begin{itemize}
\item[$\bullet$] For any $\e, \delta_{0}>0$, 
as $\delta_{\nu}$ tends to $0$, 
$$
u_{\e, \delta_{\nu}} \text{ converges to } \alpha_{\e} 
\text{ with respect to the weak } L^2\text{-topology in } 
L^{n,q}_{(2)}(F)_{\e,  \delta_{0}}. 
$$
\item[$\bullet$] For any $\e>0$, we have 
$$
\| \alpha_{\e} \|_{H_{\e},\omega}
\leq \varliminf_{ \delta_{0} \to 0}\| \alpha_{\e} \|_{\e, \delta_{0}}
\leq \varliminf_{\delta_{\nu} \to 0}\| u_{\e, \delta_{\nu}} \|_{\e,\delta_{\nu}}
\leq \|u\|_{H, \omega}. 
$$
\end{itemize}
\end{prop}

\begin{rem}\label{limit-rem} 
The subsequence $\{\delta_{\nu}\}_{\nu=1}^{\infty}$ 
does not depend on $\e, \delta_{0}$.  
The $F$-valued form $\alpha_{\e}$ is independent of $\delta_{0}$ 
and $L^2$-integrable with respect to $H_{\e}$, $\omega$ (not $\omega_{\e, \delta}$). 
\end{rem}

\begin{proof}
For any $\e, \delta_{0}>0$, by taking $\delta$ with $\delta<\delta_{0}$, 
we have 
\begin{align}\label{ineq-2}
\|u_{\e, \delta}\|_{\e, \delta_{0}}
\leq \|u_{\e, \delta}\|_{\e, \delta}
\leq \|u\|_{\e, \delta} 
\leq \|u\|_{H, \omega}. 
\end{align}
The first inequality follows from 
$\omega_{\e, \delta} \leq \omega_{\e, \delta_{0}}$ and Lemma \ref{mul}, 
the second inequality follows since $u_{\e, \delta}$ is 
the orthogonal projection of $u$ with respect to $\e, \delta$, 
and the last inequality follows from inequality $(\ref{ineq-1})$. 
From this estimate, we know that 
$\{u_{\e, \delta}\}_{\delta>0}$ is uniformly bounded 
in $L^{n,q}_{(2)}(F)_{\e,  \delta_{0}}$.
Therefore, there exist a subsequence $\{\delta_{\nu}\}_{\nu=1}^{\infty}$ of 
$\{\delta\}_{\delta>0}$ and 
$ \alpha_{\e, \delta_{0}} \in L^{n,q}_{(2)}(F)_{\e, \delta_{0}}$ such that 
$u_{\e, \delta_{\nu}}$ converges to 
$ \alpha_{\e, \delta_{0}}$ 
with respect to the weak $L^2$-topology in $L^{n,q}_{(2)}(F)_{\e,  \delta_{0}}$.  
The choice of this subsequence $\{\delta_{\nu}\}_{\nu=1}^{\infty}$ 
may depend on $\e, \delta$, 
but by extracting a suitable subsequence, 
we can easily choose a subsequence independent of $\e, \delta_{0}$ 
by Cantor's diagonal argument.

Now we prove that $ \alpha_{\e, \delta_{0}}$ does not depend on $\delta_{0}$. 
For arbitrary $\delta', \delta''$ with $0 < \delta'\leq \delta''$, 
the natural inclusion $L^{n,q}_{(2)}(F)_{\e,  \delta'}\to 
L^{n,q}_{(2)}(F)_{\e,  \delta''}$ is a bounded operator 
(continuous linear map) 
by $\|\bullet \|_{\e,  \delta''} \leq \|\bullet \|_{\e,  \delta'}$ 
(see Lemma \ref{mul}). 
By Lemma \ref{weak-hil}, the $F$-valued form 
$u_{\e, \delta_{\nu}}$ weakly converges to $ \alpha_{\e, \delta'}$ 
in not only $L^{n,q}_{(2)}(F)_{\e,  \delta'}$ but also 
$L^{n,q}_{(2)}(F)_{\e,  \delta''}$. 
Hence we have $ \alpha_{\e, \delta'}= \alpha_{\e, \delta''}$
since the weak limit is uniquely determined.

Finally we prove the estimate in the proposition. 
It is easy to see that 
$$
\| \alpha_{\e} \|_{\e, \delta_{0}}
\leq \varliminf_{\delta_{\nu} \to 0}\| u_{\e, \delta_{\nu}} \|_{\e,\delta_{0}}
\leq \varliminf_{\delta_{\nu} \to 0}\| u_{\e, \delta_{\nu}} \|_{\e,\delta_{\nu}}
\leq \|u\|_{H, \omega}. 
$$
The first inequality follows since 
the norm is lower semi-continuous with respect to the weak convergence, 
the second inequality follows from 
$\omega_{\e, \delta_{0}} \geq \omega_{\e, \delta_{\nu}}$, 
and the last inequality follows from inequality (\ref{ineq-2}). 
Fatou's lemma yields 
\begin{align}\label{Fatou}
\| \alpha_{\e} \|^2_{H_{\e}, \omega} 
=\int_{Y_{\e}} | \alpha_{\e} |^2_{H_{\e}, \omega} 
dV_{\omega}
\leq \varliminf_{\delta_{0} \to 0} 
\int_{Y_{\e}} | \alpha_{\e} |^2_{H_{\e}, \omega_{\e, \delta_{0}}} 
dV_{\omega_{\e, \delta_{0}}}
=\varliminf_{\delta_{0} \to 0} \| \alpha_{\e} \|^2_{\e, \delta_{0}}. 
\end{align} 
These inequalities lead to the estimate in the proposition. 
\end{proof}

For simplicity, 
we use the same notation $u_{\e, \delta}$ 
for the subsequence $u_{\e, \delta_{\nu}}$ in Proposition \ref{limit}.  
Next we take a suitable limit of $\alpha_{\e}$. 
For a fixed positive number $\e_{0}>0$, 
by taking a sufficiently small $\e$, 
we have 
\begin{equation}\label{estimate}
\| \alpha_{\e} \|_{H_{\e_{0}}, \omega } 
\leq \| \alpha_{\e} \|_{H_{\e}, \omega }
\leq \|u\|_{H,\omega} 
\end{equation}
by property (b) and Proposition \ref{limit}. 
By taking a subsequence of $\{\alpha_{\e} \}_{\e>0}$, 
we may assume that $\{\alpha_{\e} \}_{\e>0}$ 
weakly converges to some $\alpha$ in 
$L^{n,q}_{(2)}(F)_{H_{\e_{0}},\omega}$. 
The following proposition says that 
the proof of Theorem \ref{main} is completed  
if the weak limit $\alpha$ is shown to be zero. 

\begin{prop}\label{finish}
If the weak limit $\alpha$ is zero in 
$L^{n,q}_{(2)}(F)_{H_{\e_{0}},\omega}$, 
then the cohomology class $A$ is zero in 
$H^{q}(X, K_{X}\otimes F \otimes \I{h})$. 
\end{prop}
\begin{proof}
First we consider the De Rham-Weil isomorphism constructed in 
Proposition \ref{DW-iso}. 
$$
\begin{CD}
\dfrac{{\rm{Ker}}\, \dbar}{{\rm{Im}}\, \dbar} 
\text{ of }  L^{n,q}_{(2, {\rm{loc}})}(F)_{\e, \delta_{0}}
\xrightarrow[ \quad \phi_{1} \quad ]{\cong} 
\check{H}^{q}(X, K_{X}\otimes F \otimes \I{h_{\e}})
=\check{H}^{q}(X, K_{X}\otimes F \otimes \I{h_{}}). 
\end{CD}
$$
We remark that the $\rm{\check{C}}$ech cohomology group 
does not depend on $\e$ 
by property (c). 
By Proposition \ref{closed}, 
the subspace ${\rm{Im}}\,\dbar$ is closed in 
$L^{n,q}_{(2, {\rm{loc}})}(F)_{\e, \delta_{0}}$. 
Hence, for every $\delta$ with $0< \delta \leq \delta_{0}$, 
we can easily see that 
\begin{align}\label{closedness}
u-u_{\e,\delta} \in \overline{{\rm{Im}}\,\dbar} \text{ in } 
L^{n,q}_{(2)}(F)_{\e, \delta}
&\subset  \overline{{\rm{Im}}\,\dbar} \text{ in } 
L^{n,q}_{(2)}(F)_{\e, \delta_{0}} \\ \notag
&\subset \overline{\rm{Im}\,\dbar} 
\text{ in }  L^{n,q}_{(2, {\rm{loc}})}(F)_{\e, \delta_{0}} 
= {{\rm{Im}}\,\dbar} 
\text{ in }  L^{n,q}_{(2, {\rm{loc}})}(F)_{\e, \delta_{0}}
\end{align} 
by the construction of $u_{\e,\delta}$ and Proposition \ref{closed}. 
As $\delta$ tends to zero, we obtain 
\begin{align*}
u - \alpha_{\e} & \in \overline{\rm{Im}\, \dbar} \text{ in }  
L^{n,q}_{(2)}(F)_{\e, \delta_{0}} 
\subset \overline{\rm{Im}\,\dbar} 
\text{ in }  L^{n,q}_{(2, {\rm{loc}})}(F)_{\e, \delta_{0}} 
= {{\rm{Im}}\,\dbar} 
\text{ in }  L^{n,q}_{(2, {\rm{loc}})}(F)_{\e, \delta_{0}}. 
\end{align*}
by Lemma \ref{weak-clo} and Proposition \ref{limit}.

On the other hand, 
we have the following commutative diagram\,$:$
\[\xymatrix{
&{\rm{Ker}}\, \dbar \text{ in } L^{n,q}_{(2)}(F)_{\e, \delta_{0}}  
\ar[r]^{q_{1}\ \ }&  
\dfrac{{\rm{Ker}}\, \dbar}{{\rm{Im}}\, \dbar} 
\text{ in }  L^{n,q}_{(2, {\rm{loc}})}(F)_{\e, \delta_{0}} 
\ar[r]^{\cong\ \ \ }_{\phi_{1}\ \ }
&\check{H}^{q}(X, K_{X}\otimes F \otimes \I{h})\\
&{\rm{Ker}}\, \dbar \text{ in } L^{n,q}_{(2)}(F)_{H_{\e}, \omega}
\ar[u]_{j_{1}} \ar[r]^{j_{2}\ \ } &{\rm{Ker}}\, \dbar \text{ in } L^{n,q}_{(2)}(F)_{H_{\e_{0}}, \omega} \ar[r]^{q_{2}\ \ }  
&\dfrac{{\rm{Ker}}\, \dbar}{{\rm{Im}}\, \dbar} 
\text{ in } L^{n,q}_{(2, {\rm{loc}})}(F)_{H_{\e_{0}}, \omega}
\ar[u]^{\cong}_{\phi_2}. 
}\]
Here $j_{1}$, $j_{2}$ are the natural inclusions, 
$q_{1}$, $q_{2}$ are the natural quotient maps via 
the local $L^2$-spaces, 
and $\phi_1$, $\phi_2$ are the De Rham-Weil isomorphisms.
We remark that $j_{2}$ is well-defined. 
Indeed, by the $L^2$-integrability 
and \cite[LEMME 6.9]{Dem82}, 
the equality $\dbar U=0$ can be extended from  
$Y_{\e}$ to $X$ (in particular to $Y_{\e_{0}}$) 
for $U\in {\rm{Ker}}\, \dbar \subset L^{n,q}_{(2)}(F)_{H_{\e}, \omega}$. 
The key point here is 
the $L^2$-integrability with respect to $\omega$ 
(not $\omega_{\e,\delta}$). 
By Proposition \ref{DWiso}, the map $q_{2}$ is a compact operator, 
and thus we obtain 
$$
\lim_{\e \to 0}q_{2}(u-\alpha_{\e})
=q_{2}(u-\alpha)=q_{2}(u)   
$$
by Lemma \ref{comp-hil} and the assumption $\alpha =0$. 
On the other hand, we can see that 
$q_{1}(u-\alpha_{\e})=0$
by the first half argument. 
Therefore we obtain $q_{2}(u)=0$ by the above commutative diagram. 
Then we can conclude that $u$ belongs to ${\rm{Im}}\, \dbar$ in 
$L^{n,q}_{(2, {\rm{loc}})}(F)_{H, \omega}$. 
Indeed, we can easily see that $q_{3}(u)=0$ 
by the following commutative diagram\,:
\[\xymatrix{
&{\rm{Ker}}\, \dbar \text{ in } L^{n,q}_{(2)}(F)_{H_{\e_{0}},\omega}  
\ar[r]^{q_{2}\ \ }&  
\dfrac{{\rm{Ker}}\, \dbar}{{\rm{Im}}\, \dbar} 
\text{ of }  L^{n,q}_{(2, {\rm{loc}})}(F)_{H_{\e_{0}},\omega} 
\ar[r]^{\cong\ \ \ }_{\phi_{2}\ \ }
&\check{H}^{q}(X, K_{X}\otimes F \otimes \I{h_{\e_{0}}})\\
&{\rm{Ker}}\, \dbar \text{ in } L^{n,q}_{(2)}(F)_{H, \omega}
\ar[u]_{j_{3}} \ar[r]^{q_{3}\ \ } 
&\dfrac{{\rm{Ker}}\, \dbar}{{\rm{Im}}\, \dbar} 
\text{ of }  L^{n,q}_{(2, {\rm{loc}})}(F)_{H, \omega}  
\ar[r]^{\cong\ \ \ }_{\phi_{3}\ \ }  
& \check{H}^{q}(X, K_{X}\otimes F \otimes \I{h_{\e}}). 
\ar@{=}[u]
}\]
\end{proof}
\end{step}

\begin{step}[Relations between weak limits and $L^2$-norms]\label{S3}
In this step, we consider the norm 
$$\|su_{\e, \delta}\|_{\e, \delta}:= 
\|su_{\e, \delta}\|_{H_{\e}h_{\e}^{m}, \omega_{\e,\delta}}$$ 
and prove Proposition \ref{finish2}, 
which says that it is sufficient for the proof  
to show that 
$$\varliminf_{\e \to 0} \varliminf_{\delta \to 0} \|su_{\e, \delta}\|_{K, \e, \delta} =0$$
for every relatively compact set $K \Subset X$. 

In order to clarify a relation between 
the weak limit $\alpha$ and the 
asymptotic behavior of the norm of $su_{\e, \delta}$, 
we compare the norm of $u_{\e, \delta}$ with the norm of $su_{\e, \delta}$. 
We define $Y^{k}_{\e_{0}}$ and $X_{c}$ by  
\begin{align*}
Y^{k}_{\e_{0}}:=\{y\in Y_{\e_{0}} \, 
|\, |s|_{h^{m}_{\e_{0}}}(y)>1/k\}
\quad \text{ and }\quad  
X_{c}:=\{x\in X \, |\, \Phi(x)<c\}
\end{align*}
for $k \gg 0$ and $c$. 
The subset $X_{c}$ is a relatively compact set in 
$X$ for every $c$ with $c < \sup_{X} \Phi$ by the construction of $\Phi$. 
Further $Y^{k}_{\e_{0}}$ is an open set in $Y_{\e_{0}}$ since 
$ |s|_{h_{\e_{0}}^{m}}$ is lower semi-continuous. 
Then we prove the following proposition\,:

\begin{prop}\label{finish2}
Under the above situation, 
if we have 
$$\varliminf_{\e \to 0} \varliminf_{\delta \to 0} \|su_{\e, \delta}\|_{X_{c}, \e, \delta} =0$$
for every $c$ with $c < \sup_{X} \Phi$, 
then the weak limit $\alpha$ is zero. 
In particular, 
the cohomology class $A$ is zero by Proposition \ref{finish}. 
\end{prop}

\begin{proof}
By the argument on inequality (\ref{estimate}), 
we are assuming that $\alpha_{\e}$ weakly converges 
to $\alpha$ 
in $L^{n,q}_{(2)}(F)_{H_{\e_{0}}, \omega}$. 
The restriction 
$\alpha_{\e}|_{X_{c}\cap Y^{k}_{\e_{0}}}$ also weakly converges 
to $\alpha|_{X_{c}\cap Y^{k}_{\e_{0}}}$ 
in $L^{n,q}_{(2)}(X_{c}\cap Y^{k}_{\e_{0}},F)_{H_{\e_{0}}, \omega}$ 
by Lemma \ref{weak-hil}, 
since the restriction map 
$$
L^{n,q}_{(2)}(F)_{H_{\e_{0}}, \omega} 
\longrightarrow 
L^{n,q}_{(2)}(X_{c}\cap Y^{k}_{\e_{0}},F)_{H_{\e_{0}}, \omega}
$$ 
is a bounded operator (continuous linear map). 
Therefore we obtain 
\begin{align*}
\| \alpha \|_{X_{c} \cap Y^{k}_{\e_{0}}, H_{\e_{0}}, \omega}
\leq \varliminf_{\e \to 0} 
\| \alpha_{\e} \|_{X_{c} \cap Y^{k}_{\e_{0}}, H_{\e_{0}}, \omega}
\leq \varliminf_{\e \to 0} 
\| \alpha_{\e} \|_{X_{c} \cap Y^{k}_{\e_{0}}, H_{\e}, \omega}. 
\end{align*}
The first inequality follows since 
the norm is lower semi-continuous with respect to the weak convergence, 
and  the second inequality follows from property (b). 
By the same argument, 
the restriction of $u_{\e,\delta}$ weakly converges to 
$\alpha_{\e}$ 
in $L^{n,q}_{(2)}(X_{c}\cap Y^{k}_{\e_{0}},F)_{\e, \delta_{0}}$, 
and thus  we obtain 
\begin{align*}
\| \alpha_{\e} \|_{X_{c} \cap Y^{k}_{\e_{0}},\e, \delta_{0}}
\leq \varliminf_{\delta \to 0} 
\| u_{\e,\delta} \|_{X_{c} \cap Y^{k}_{\e_{0}}, \e, \delta_{0}}
\leq \varliminf_{\delta \to 0} 
\| u_{\e,\delta} \|_{X_{c} \cap Y^{k}_{\e_{0}}, \e, \delta}. 
\end{align*}
Moreover, we can obtain 
$$
\| \alpha_{\e} \|_{X_{c} \cap Y^{k}_{\e_{0}},H_{\e}, \omega}
\leq 
\varliminf_{\delta_{0} \to 0}
\| \alpha_{\e} \|_{X_{c} \cap Y^{k}_{\e_{0}},\e, \delta_{0}}
\leq \varliminf_{\delta \to 0} 
\| u_{\e,\delta} \|_{X_{c} \cap Y^{k}_{\e_{0}}, \e, \delta}
$$ 
by the above inequality and 
Fatou's lemma (see the argument for inequality (\ref{Fatou})). 
These inequalities yield    
\begin{align*}
\| \alpha \|_{X_{c} \cap Y^{k}_{\e_{0}}, H_{\e_{0}}, \omega}
\leq \varliminf_{\e \to 0} \varliminf_{\delta \to 0}
\| u_{\e,\delta} \|_{X_{c} \cap Y^{k}_{\e_{0}}, 
\e, \delta}. 
\end{align*}
On the other hand, 
from 
$1/k < |s|_{h^{m}_{\e_{0}}}\leq |s|_{h^{m}_{\e}}$ 
on $Y^{k}_{\e_{0}}$, 
we have 
\begin{align*}
\| u_{\e,\delta} \|_{X_{c} \cap Y^{k}_{\e_{0}}, 
\e, \delta}
\leq k \| su_{\e,\delta} \|_{X_{c} \cap Y^{k}_{\e_{0}}, 
\e, \delta}
\leq k \| su_{\e,\delta} \|_{X_{c}, \e, \delta}. 
\end{align*}
By the assumption, 
we can conclude that $\alpha=0$ on $X_{c} \cap Y^{k}_{\e_{0}}$ 
for arbitrary $c < \sup \Phi$ and $k \gg 0$. 
From $\cup_{\sup \Phi>c, k\gg 0 }(X_{c} \cap Y^{k}_{\e_{0}})=Y_{\e_{0}}$, 
we obtain the conclusion. 
\end{proof}
\end{step}

\begin{step}[Construction of solutions of the $\dbar$-equation]\label{S4}

In this step, 
by using the construction of the De Rham-Weil isomorphism in subsection \ref{Sec2-6}, 
we prove Proposition \ref{sol},  
which gives a solution $w_{\e, \delta}$ of the $\dbar$-equation 
$\dbar w_{\e, \delta}=u-u_{\e, \delta}$ 
with uniformly bounded (local) $L^2$-norm.

Fix a locally finite open cover $\mathcal{U}:=\{B_{i} \}_{i \in I}$ of $X$ by 
sufficiently small Stein open sets $B_{i} \Subset X$.
Since $h_{\e}$, $\omega_{\e,\delta}$, and $Y_{\e}$ satisfy 
the assumptions in Proposition \ref{DW-iso}, 
we have the continuous maps  
\begin{align*}
f_{\e,\delta}: {\rm{Ker}}\, \dbar \text{ in } L^{n,q}_{(2, {\rm{loc}})}(F)_{\e,\delta} 
\rightarrow 
{\rm{Ker}}\, \mu \text{ in } 
C^{q}(\mathcal{U}, K_{X} \otimes F \otimes \I{h_{\e}}), \\
g_{\e. \delta}: {\rm{Ker}}\, \mu \text{ in } 
C^{q}(\mathcal{U}, K_{X} \otimes F \otimes \I{h_{\e}})
\rightarrow 
{\rm{Ker}}\, \dbar \text{ in } L^{n,q}_{(2, {\rm{loc}})}(F)_{\e,\delta}
\end{align*}
such that they determine the isomorphism between the $\dbar$-cohomology 
and the $\rm{\check{C}}$ech cohomology. 
For the construction of $f_{\e,\delta}$ in Proposition \ref{DW-iso}, 
we locally solved the $\dbar$-equation and 
estimated the $L^{2}$-norm of the solution by Lemma \ref{LL}. 
In this subsection, for the $L^{2}$-estimate of the solution, 
we use the following lemma instead of Lemma \ref{LL}

\begin{lemm}\label{loc-sol}
Let $B \Subset X$ be a sufficiently small Stein open set. 
Then, for an arbitrary  
$U \in {\rm{Ker}}\, \dbar \subset 
L^{n,q}_{(2)}(B\setminus Z_{\e}, F)_{\e, \delta}$, 
there exist $V \in L^{n,q-1}_{(2)}(B\setminus Z_{\e}, F)_{\e,\delta}$ 
and a positive constant $C_{\e,\delta}$ 
$($depending only on $\Psi_{\e,\delta}$, $q$$)$
such that 
\begin{itemize}
\item[$\bullet$] $\dbar V=U \text{  and } \|V\|_{\e, \delta} \leq 
C_{\e,\delta} \|U\|_{\e, \delta}$. 
\item[$\bullet$] $\lim_{\delta \to 0} C_{\e, \delta}$ is independent of $\e$. 
\end{itemize}
\end{lemm}

\begin{proof}
We may assume that $\e < 1/2$. 
Further, by property (C), we may assume that 
there exists a bounded function $\Psi_{\e,\delta}$ on $B$ such that 
$\omega_{\e,\delta}=\deldel \Psi_{\e,\delta}$ and 
$\Psi_{\e,\delta} \to \Psi$ as $\delta \to 0$. 
The function $\Psi$ is independent of $\e$ 
since it is the local weight function of $\omega$. 
The curvature of $G_{\e,\delta}$ defined by 
$G_{\e,\delta}:=H_{\e}e^{-\Psi_{\e,\delta}}$
satisfies 
\begin{align*}
\sqrt{-1}\Theta_{G_{\e,\delta}}(F)
&= \sqrt{-1}\Theta_{h_{\e}}(F)+\deldel \chi(\Phi)+\deldel \Psi_{\e,\delta} \\
&\geq -\e\omega+\omega_{\e, \delta} \\
&\geq (1-\e)\omega_{\e, \delta} 
\end{align*}
by property  (a) and property (B). 
Here we used the inequality $\deldel \chi(\Phi) \geq 0$, 
which follows 
since $\Phi$ is a psh function and $\chi$ is an increasing convex function.   
It follows that $\|U\|_{G_{\e,\delta}, \omega_{\e,\delta}}$ is finite 
since $\Psi_{\e,\delta}$ is a bounded function. 
Hence, by the standard result for the $\dbar$-equation, 
there exist 
a solution $V$ 
such that  $\dbar V =U$  and 
$ \|V\|^2_{G_{\e,\delta}, \omega_{\e,\delta}} 
\leq (1/q(1-\e)) \|U\|^2_{G_{\e,\delta},\omega_{\e,\delta}}$. 
By $(1-\e)>1/2$ and the definition of $G_{\e,\delta}$, 
we can easily see that 
$$
\|V\|_{\e, \delta} \leq 
\sqrt{\frac{2}{q}}\, \dfrac{\sup_{B} e^{-\Psi_{\e, \delta}}}
{ \inf_{B}  e^{-\Psi_{\e, \delta}}}
\|U\|_{\e, \delta}.
$$
The above constant converges to $(2/q)^{1/2}$ as $\delta$ tends to zero. 
\end{proof}

\begin{prop}\label{sol}
For every $c$ with $c < \sup_{X}\Phi$, there exists   
$w_{\e,\delta} \in L^{n,q-1}_{(2, {\rm{loc}})}(F)_{\e,\delta}$ 
with the following properties\,$:$ 
\begin{itemize}
\item[$\bullet$] $\dbar  w_{\e,\delta}=u-u_{\e, \delta}$. 
\item[$\bullet$] $\varlimsup_{\delta \to 0} \|w_{\e,\delta}\|_{X_{c}, \e, \delta} $ can be bounded by a constant independent of $\e$. 
\end{itemize}
\end{prop}

\begin{rem}\label{rem-sol}
We have already known that 
there exists a solution $w_{\e,\delta}$ of the $\dbar$-equation 
$\dbar w_{\e,\delta}=u-u_{\e, \delta} $ 
by $u-u_{\e, \delta} \in {\rm{Im}\,\dbar} \subset 
L^{n,q}_{(2, {\rm{loc}})}(F)_{\e, \delta}$ 
(see (\ref{closedness}) in the proof of Proposition \ref{finish}).
The important point here is 
the second property on the local $L^2$-norm of solutions.
\end{rem}

The strategy of the proof is the same as in the proof of 
\cite[Proposition 5.9]{FM16} and \cite[Theorem 5.9]{Mat13}. 
The main idea is to change the $\dbar$-equation 
$\dbar w_{\e,\delta}=u-u_{\e, \delta}$
to the equation $\mu \gamma_{\e,\delta} = f_{\e,\delta}(u-u_{\e,\delta})$ of 
the coboundary operator $\mu$
in the set of cochains 
$C^{\bullet}(K_{X}\otimes F \otimes \I{h_{\e}})$, 
by using the $\rm{\check{C}}$ech complex and pursuing  
the De Rham-Weil isomorphism. 
(A similar argument can be found in \cite{Ohs84}.)
Here $f_{\e,\delta}$ is the continuous map 
constructed in Proposition \ref{DW-iso}.  
The $L^{2}$-space 
$L^{n,q}_{(2)}(F)_{\e, \delta}$ depends on $\e, \delta$, 
but $C^{\bullet}(K_{X}\otimes F \otimes \I{h_{\e}})$ 
does not depend on them thanks to property (c). 
This is one of the important points.  
In the proof, 
we will show that $f_{\e,\delta}(u-u_{\e,\delta})$ converges to 
some $q$-coboundary $\alpha_{0,0}$ in $C^{q}(K_{X}\otimes F \otimes \I{h})$ 
with respect to the topology induced 
by the local $L^{2}$-norms $\{p_{K_{i_{0}...i_{q}}}(\bullet)\}_{K_{i_{0}...i_{q}} \Subset B_{i_{0}...i_{q}}}$
(see (\ref{semi-defi}) for the definition).
Further we will show that 
the coboundary operator $\mu$ is an open map.
Then, by these observations, we will construct 
a solution $\gamma_{\e,\delta}$ of 
the equation $\mu \gamma_{\e,\delta} = f_{\e,\delta}(u-u_{\e,\delta})$ 
with suitable local $L^{2}$-norm. 
Finally, by the continuous map $g_{\e,\delta}$ constructed by a partition of unity, 
we conversely construct $w_{\e,\delta}$ 
satisfying the properties in Proposition \ref{sol}.

For the reader's convenience, 
we first give a proof for the case $q=1$. 
This case helps us to follow the essential arguments.

\begin{proof}[Proof of Proposition \ref{sol} for the case $q=1$]
We may assume that the cardinality of $I_{c}$ defined by 
$$
I_{c}:=\{i \in I \, | \, B_{i} \cap X_{c}\not=\emptyset \}
$$ 
is finite by $X_{c} \Subset X$. 
For simplicity we put $U_{\e, \delta}:=u-u_{\e, \delta}$. 
By Lemma \ref{loc-sol}, we can take a solution $\beta_{\e, \delta,i}$
of the $\dbar$-equation 
$\dbar \beta_{\e, \delta,i}= U_{\e, \delta}$ on $B_{i} \setminus Z_{\e}$
such that  
$$
\|\beta_{\e, \delta,i}\|_{B_{i}, \e, \delta} 
\leq C_{\e, \delta}\|U_{\e, \delta} \|_{B_{i}, \e, \delta}
\leq C_{\e, \delta}\|U_{\e, \delta} \|_{\e, \delta}
$$ 
for some constant $C_{\e,\delta}$.  
In the proof, the notation $C_{\e,\delta}$ denotes 
a (possibly different) positive constant 
with the property in Lemma \ref{loc-sol} 
(that is, $\lim_{\delta \to 0}C_{\e,\delta}$ is 
independent of $\e$). 
Inequality (\ref{ineq-2}) yields 
\begin{align*}
\|U_{\e, \delta} \|_{\e, \delta} 
\leq \|u \|_{\e, \delta}+\|u_{\e, \delta} \|_{\e, \delta}
\leq 2\|u \|_{H, \omega}. 
\end{align*}
In particular, the norm $\|\beta_{\e, \delta,i}\|_{B_{i}, \e, \delta}$ 
can be bounded by a constant $C_{\e,\delta}$.

Now we consider the $F$-valued $(n,0)$-form 
$(\beta_{\e,\delta, j}-\beta_{\e, \delta,i})$ on $B_{ij} \setminus Z_{\e}$, 
where $B_{ij}:=B_{i} \cap B_{j}$. 
Then  $(\beta_{\e,\delta, j}-\beta_{\e, \delta,i})$ can be seen 
as a holomorphic function with bounded $L^2$-norm,  
since it is a $\dbar$-closed $F$-valued $(n,0)$-form 
satisfying $\| \beta_{\e,\delta, j}-\beta_{\e, \delta,i} \|_{H_{\e}, \omega}=
\| \beta_{\e,\delta, j}-\beta_{\e, \delta,i} \|_{\e,\delta} < \infty$ 
(see Lemma \ref{mul}). 
By the Riemann extension theorem, 
it can be extended to the $F$-valued $(n,0)$-form on $B_{ij}$ 
(for which we use same the notation). 
Further it belongs to 
$H^{0}(B_{ij},K_{X}\otimes F \otimes \I{h})$ by property (c). 
Note  that we can use  property (c) 
thanks to a special property of $(n,0)$-forms 
(holomorphic functions).

We define the $1$-cocycle $\alpha_{\e,\delta}$ by 
$$
\alpha_{\e,\delta}:= \mu \{\beta_{\e,\delta, i}\}
=\{ \beta_{\e,\delta, j}-\beta_{\e, \delta,i}\}, 
$$
where $\mu$ is the coboundary operator.
The topology of $C^{q}(\mathcal{U}, K_{X}\otimes F \otimes \I{h})$ 
is induced by the semi-norms $\{p_{K}(\bullet)\}_{K \Subset B_{i_{0}...i_{q}}}$ defined 
to be  
$$
p_{K}^2(\{\alpha_{i_{0}...i_{q}}\}):=
\int_{K} |\alpha_{i_{0}...i_{q}}|^{2}_{H, \omega}\, dV_{\omega}
$$
for every $\{\alpha_{i_{0}...i_{q}}\} 
\in C^{q}(\mathcal{U}, K_{X}\otimes F \otimes \I{h})$ 
and $K\Subset U_{i_{0}...i_{q}}$. 
The above integrand is independent of $\omega$ 
since $\alpha_{i_{0}...i_{q}}$ is an $F$-valued $(n,0)$-form (see Lemma \ref{mul}). 
Then $C^{p}(\mathcal{U}, K_{X}\otimes F \otimes \I{h})$ becomes  
a Fr$\rm{\acute{e}}$chet space with respect to these semi-norms 
by Lemma \ref{Fre}. 
Then we prove the following claim\,:

\begin{claim}\label{con}
There exist subsequences $\{\e_{k}\}_{k=1}^{\infty}$ and 
$\{\delta_{\ell}\}_{\ell=1}^{\infty}$ with the following properties\,$:$
\begin{itemize}
\item[$\bullet$] $\alpha_{\e_{k},\delta_{\ell}} \to \alpha_{\e_{k},0}$ 
in $C^{1}(\mathcal{U}, K_{X}\otimes F \otimes \I{h})$
as 
$\delta_{\ell} \to 0$.
\item[$\bullet$] $\alpha_{\e_{k},0} \to \alpha_{0,0}$ in $C^{1}(\mathcal{U}, K_{X}\otimes F \otimes \I{h})$ as $\e_{\ell} \to 0$.
\end{itemize}
\end{claim}
\begin{proof}[Proof of Claim \ref{con} ]

We regard $\alpha_{\e,\delta, ij}:=\beta_{\e,\delta, j}-\beta_{\e, \delta,i}$
as a holomorphic function on $B_{ij}$. 
By the construction of $\beta_{\e, \delta,i}$, 
the norm $\|\alpha_{\e,\delta, ij}\|_{B_{ij}, \e, \delta}$ 
can be bounded by a constant $C_{\e, \delta}$. 
This implies that the sup-norm  $\sup_{K}|\alpha_{\e,\delta, ij}|$ is also 
uniformly bounded with respect to $\delta$
for every $K \Subset B_{ij}$.
(Recall the local sup-norm of holomorphic functions 
can be estimated by the $L^{2}$-norm). 
By Montel's theorem, we can take a 
subsequence $\{\delta_{\ell}\}_{\ell=1}^{\infty}$ 
with the first property. 
Then the norm of the limit $\alpha_{\e,0}$ can be bounded by 
a positive constant independent of $\e$ 
since $\lim_{\delta \to 0}C_{\e,\delta}$ is independent of $\e$. 
Thus we can take a subsequence $\{\e_{k}\}_{k=1}^{\infty}$ 
with the second properties. 
The convergence with respect to the local sup-norms 
implies the  the convergence 
with respect to the local $L^2$-norms $\{p_{K}(\bullet)\}_{K \Subset B_{i_{0}...i_{q}}}$ 
(for example see \cite[Lemma 5.2]{Mat13}). 
This completes the proof. 
\end{proof}

For simplicity, we continue to use the same notation for the subsequence 
in Claim \ref{con}.

\begin{claim}\label{clo}
The cocycle $\alpha_{\e,\delta}$ is a coboundary. 
In particular the limit $\alpha_{0,0}$ is also a coboundary. 
\end{claim}
\begin{proof}[Proof of Claim \ref{clo} ]
By Remark \ref{rem-sol}, we can see that 
$U_{\e, \delta}=u-u_{\e, \delta}$ belongs to 
${\rm{Im}\,\dbar}$ in $L^{n,q}_{(2, {\rm{loc}})}(F)_{\e,\delta}$. 
Further, by the isomorphism in Proposition \ref{DW-iso}, 
we can see that $\alpha_{\e,\delta}$ is a coboundary. 
Since we are assuming that $X$ is holomorphically convex (see Step \ref{Sec1}), 
the set of $q$-coboundaries 
$B^{q}(\mathcal{U}, K_{X} \otimes F \otimes \I{h})= {{\rm Im}}\,\mu$ is 
a Fr\'echet space by Lemma \ref{Fre}. 
Therefore we obtain the latter conclusion. 
\end{proof}

We will construct a solution $\gamma_{\e,\delta}$ of the $\mu$-equation 
$\mu \gamma_{\e,\delta} = \alpha_{\e,\delta}$ 
with suitable local $L^2$-norm. 
The coboundary operator 
$$
\mu\, \colon\, C^{q-1}(\mathcal{U}, K_{X}\otimes F \otimes \I{h})
\to B^{q}(\mathcal{U}, K_{X}\otimes F \otimes \I{h})
$$
is continuous and surjective between Fr$\rm{\acute{e}}$chet spaces, 
and thus it is an open map by the open mapping theorem. 
From the latter conclusion of Claim \ref{clo}, there exists  
$\gamma_{0,0}\in C^{0}(\mathcal{U}, K_{X}\otimes F \otimes \I{h})$ 
such that 
$\mu \gamma_{0,0}=\alpha_{0,0}$. 
For an arbitrary family $K:=\{K_{i}\}_{i\in I_{c}}$ of 
relative compact sets $K_{i} \Subset B_{i}$, 
the image $\mu (\Delta_{K})$ of $\Delta_{K}$ is an open neighborhood 
of $\alpha_{0,0}$, 
where $\Delta_{K}$ is an open neighborhood of $\gamma_{0,0}$ defined by 
\begin{equation}\label{delta}
\Delta_{K}:=\{\gamma \in C^{0}(\mathcal{U}, K_{X}\otimes F \otimes \I{h})
\,\mid\, p_{K_{i}}(\gamma-\gamma_{0,0})<1\,\text{ for every $i\in I_{c}$}\}. 
\end{equation}
Since the image $\mu (\Delta_{K})$ is an open neighborhood of $\alpha_{0,0}$ 
and $\alpha_{\e,\delta}$ converges to $\alpha_{0,0}$, 
we can take $\gamma_{\e,\delta}=\{ \gamma_{\e,\delta, i} \}\in \Delta_{K}$ such that  
\begin{align}
\label{eq-sol} \{\gamma_{\e,\delta,j}-\gamma_{\e,\delta, i}\}&= \mu \gamma_{\e,\delta}=\alpha_{\e,\delta}
=\{\beta_{\e,\delta,j}-\beta_{\e,\delta, i}\}, \\
\label{eq-est}  p^2_{K_{i}}(\gamma_{\e,\delta})&=
\int_{K_{i}}|\gamma_{\e,\delta, i}|^{2}_{H,\omega}\, dV_{\omega} \leq C_{K} 
\text{ for every $i \in I_{c}$}
\end{align}
for some positive constant $C_{K}$ 
(depending on $K$, $\gamma$ but does not depend on $\e, \delta$). 

Let us construct a solution $w_{\e,\delta}$ 
with the properties in Proposition \ref{sol}. 
We fix a partition of unity $\{\rho_{i}\}_{i \in I}$. 
Then, by $\dbar \gamma_{\e,\delta,i}=0$ and $\dbar \beta_{\e,\delta,i}=U_{\e,\delta}$, 
we have 
\begin{align*}
\dbar \sum_{k \in I} \rho_{k}(\gamma_{\e,\delta,i}-\gamma_{\e,\delta, k})&=\dbar \sum_{k \in I} \rho_{k} \gamma_{\e,\delta, k,}\\
\dbar \sum_{k \in I} \rho_{k}(\beta_{\e,\delta,i}-\beta_{\e,\delta, k}) 
&=U_{\e,\delta}-\dbar \sum_{k \in I} \rho_{k} \beta_{\e,\delta, k}.
\end{align*}
When we define $w_{\e, \delta}$ by 
$$
w_{\e,\delta}:=
\sum_{k \in I} \rho_{k} \beta_{\e,\delta, k}+ \sum_{k \in I} \rho_{k} \gamma_{\e,\delta, k,}
$$
it is easy to check $U_{\e,\delta}=\dbar w_{\e,\delta}$ by equality (\ref{eq-sol}). 
It remains to estimate the $L^2$-norm of $w_{\e,\delta}$. 
By putting $K_{i}$ by $K_{i}:={\rm{Supp}}\, \rho_{i}$, 
we may assume that  the inequality 
$p^2_{K_{i}}(\gamma_{\e,\delta})=\int_{{\rm{Supp}}\, 
\rho_{i}}|\gamma_{\e,\delta, i}|^{2}_{H, \omega} dV_{\omega} \leq C_{K}$ holds 
for every $i \in I_{c}$
by inequality $(\ref{eq-est})$. 
Therefore we obtain  
\begin{align*}
\|\sum_{k \in I} \rho_{k}\gamma_{\e,\delta, k} \|^2_{X_{c},\e,\delta}=
\int_{X_{c}} \big| \sum_{k \in I} \rho_{k}\gamma_{\e,\delta, k} \big|^2_{H_{\e}, \omega}\, dV_{\omega}
\leq \sum_{k \in I_{c}}
\int_{B_{k} \cap\, {\rm{Supp}}\, \rho_{k}} 
\big|  \gamma_{\e,\delta, k}
 \big|^2_{H_{\e}, \omega}\, dV_{\omega}
\leq C_{K}\,\sharp I_{c}. 
\end{align*}
Note that the cardinality of $\ I_{c}$ is finite 
by the choice of $\mathcal{U}$.  
Further, we obtain 
\begin{align*}
\|\sum_{k \in I} \rho_{k}\beta_{\e,\delta, k} \|^2_{X_{c},\e,\delta}=
\int_{X_{c}} \big| \sum_{k \in I} \rho_{k}\beta_{\e,\delta, k} 
\big|^2_{\e, \delta}\, dV_{\omega_{\e, \delta}}
\leq \sum_{k \in I_{c}}
\int_{B_{k}} \big|  \beta_{\e,\delta,k} \big|^2_{\e, \delta}\, 
dV_{\omega_{\e, \delta}}
\leq C^2_{\e, \delta}\,\sharp I_{c}\, \|u\|_{H, \omega}
\end{align*}
for some $C_{\e, \delta}>0$ by the construction of $\beta_{\e,\delta,i}$. 
These inequalities lead to the desired estimate of $w_{\e,\delta}$. 
\end{proof}

\begin{proof}[Proof of Proposition \ref{sol} for the general case]
For simplicity, we put $U_{\e,\delta}:=u-u_{\e,\delta} \in \Image 
\dbar \subset L^{n,q}_{(2,{\rm{loc}})}(F)_{\e,\delta}$. 
Then there exist the $F$-valued $(n,q-k-1)$-forms 
$\beta^{\e, \delta}_{i_{0}\dots i_{k}}$ 
on $B_{i_{0}\dots i_{k}} \setminus Z_{\e}$ 
satisfying 
\[
  (*) \left\{ \quad
  \begin{array}{ll}
\vspace{0.2cm}
\dbar \{\beta^{\e, \delta}_{i_{0}}\} &=\{U_{{\e, \delta}} |_{B_{i_{0}}\setminus Z_{\e}}\},  \\
\dbar \{ \beta^{\e, \delta}_{i_{0}i_{1}} \}&=\mu  \{\beta^{\e, \delta}_{i_{0}}\},  \\
\dbar \{ \beta^{\e, \delta}_{i_{0}i_{1}i_{2}} \}&
=\mu   \{\beta^{\e, \delta}_{i_{0}i_{1}}\},  \\
 & \vdots   \\
\dbar \{ \beta^{\e, \delta}_{i_{0}\dots i_{q-1}} \}
&=\mu   \{\beta^{\e, \delta}_{i_{0}\dots i_{q-2}}\}, \\
f_{\e, \delta}(U_{\e, \delta})&=\mu 
\{\beta^{\e, \delta}_{i_{0}\dots i_{q-1}}\}. 
  \end{array} \right.
\]
Here $\beta^{\e, \delta}_{i_{0}\dots i_{k}}$ 
is the solution of the above equation whose 
norm is minimum among all the solutions 
(see the construction of $f$ in Proposition \ref{DW-iso}). 
For example, $\beta^{\e, \delta}_{i_{0}}$ is the solution 
of $\dbar \beta^{\e, \delta}_{i_{0}}=U_{{\e, \delta}}$ on  
${B_{i_{0}}\setminus Z_{\e}}$ whose norm 
$\|\beta^{\e, \delta}_{i_{0}} \|_{\e,\delta}$ 
is minimum among all the solutions. 
In particular, we have 
$\|\beta^{\e, \delta}_{i_{0}}\|^{2}_{\e,\delta} \leq 
C_{\e, \delta} \|U_{\e, \delta} \|^2_{B_{i_{0}}, \e,\delta}
\leq C_{\e, \delta} \|U_{\e, \delta} \|^2_{\e,\delta}$ 
for some constant $C_{\e,\delta}$ by Lemma \ref{loc-sol}, 
where $C_{\e,\delta}$ is a constant such that 
$\lim_{\delta \to 0} C_{\e,\delta}$ is independent of $\e$. 
Similarly, $\beta^{\e, \delta}_{i_{0}i_{1}}$ is 
the solution of $\dbar \beta^{\e, \delta}_{i_{0}i_{1}}= 
(\beta^{\e, \delta}_{i_{1}}- \beta^{\e, \delta}_{i_{0}})$ on 
$B_{i_{0}i_{1}}\setminus Z_{\e}$ 
and the norm 
$\|\beta^{\e, \delta}_{i_{0}i_{1}}\|_{\e,\delta}$
is minimum among all the solutions. 
In particular, we have 
$\|\beta^{\e, \delta}_{i_{0}i_{1}}\|_{\e,\delta} \leq 
D_{\e, \delta} \|(\beta^{\e, \delta}_{i_{1}}- \beta^{\e, \delta}_{i_{0}})\|_{\e,\delta}$ 
for some constant $D_{\e,\delta}$. 
It is easy to see that  
$$
\|\beta^{\e, \delta}_{i_{0}i_{1}}\|_{\e,\delta} \leq 
D_{\e, \delta} \|(\beta^{\e, \delta}_{i_{1}}- 
\beta^{\e, \delta}_{i_{0}})\|_{\e,\delta}
\leq 
2C_{\e, \delta}D_{\e,\delta} \|U_{\e, \delta} \|_{\e,\delta}
\leq 2C_{\e, \delta}D_{\e,\delta} \|u \|_{H, \omega}. 
$$
By repeating this process, 
we obtain 
\begin{align*}
\|\beta^{\e, \delta}_{i_{0}\dots i_{k}}\|^2_{\e, \delta}\leq 
C_{\e, \delta} \|u\|^2_{h,\omega}
\end{align*}
for a constant $C_{\e,\delta}$  such that 
$\lim_{\delta \to 0} C_{\e,\delta}$ is independent of $\e$. 
Moreover, by property (c), we obtain  
$$
\alpha_{\e, \delta}:=f_{\e, \delta}(U_{\e, \delta})=
\mu \{\beta^{\e, \delta}_{i_{0}\dots i_{q-1}}\} 
\in C^{q}(\mathcal{U}, K_X\otimes F\otimes \I{h_{\e}})
=C^{q}(\mathcal{U}, K_X\otimes F\otimes \I{h}).
$$
By the same arguments as in Claim \ref{con} and Claim \ref{clo}, 
we obtain the following\,:
\begin{claim}\label{f-claim}
There exist subsequences $\{\e_{k}\}_{k=1}^{\infty}$ and 
$\{\delta_{\ell}\}_{\ell=1}^{\infty}$ with the following properties\,$:$
\begin{itemize}
\item[$\bullet$] $\alpha_{\e_{k},\delta_{\ell}} \to \alpha_{\e_{k},0}$ 
in $C^{q}(\mathcal{U}, K_{X}\otimes F \otimes \I{h})$
as 
$\delta_{\ell} \to 0$.
\item[$\bullet$] $\alpha_{\e_{k},0} \to 
\alpha_{0,0}$ in $C^{q}(\mathcal{U}, 
K_{X}\otimes F \otimes \I{h})$ as $\e_{k} \to 0$. 
\end{itemize}
Moreover, the limit $\alpha_{0,0}$ belongs to 
$B^{q}(\mathcal{U}, K_X\otimes F \otimes \I{h})= \Image\mu$. 
\end{claim}

By the latter conclusion of the claim, 
there exists  
$\gamma_{0,0} \in C^{q-1}(\mathcal{U}, K_X\otimes F \otimes \I{h})$ 
such that $\mu \gamma_{0,0} = \alpha_{0,0}$. 
The coboundary operator 
\begin{equation*}
\mu: C^{q-1}(\mathcal{U}, K_X\otimes F\otimes \I{h}) \to 
B^{q}(\mathcal{U}, K_X\otimes F\otimes \I{h})= \Image\mu
\end{equation*}
is an open map by the open mapping theorem. 
For an arbitrary family $K:=\{K_{i}\}_{i\in I_{c}}$ of 
relative compact sets $K_{i} \Subset B_{i}$, 
we define $\Delta_{K}$ by (\ref{delta}). 
Then since $\mu (\Delta_{K})$ is an open neighborhood of 
the limit $\alpha_{0,0}$ in $\Image \mu$, 
we can obtain $\gamma_{\e, \delta} \in 
C^{q-1}(\mathcal{U}, K_X\otimes F \otimes \I{h})$ 
such that 
\begin{align*}
\mu \gamma_{\e, \delta}= \alpha_{\e, \delta} \text{\quad and \quad}
p_{K_{i_{0}...i_{q-1}}}(\gamma_{\e, \delta})^{2}\leq C_{K} 
\end{align*}
for some positive constant $C_{K}$. 
The above constant $C_{K}$ 
depends on the choice of $K$, $\gamma$, but does not depend on $\e, \delta$. 

By the same argument as in \cite[Claim 5.11 and Claim 5.13]{Mat13}, 
we can obtain $F$-valued $(n,q-1)$-forms $w_{\e, \delta}$ with the desired properties. 
The strategy is as follows:
The inverse map $\overline{g_{\e,\delta}}$ of $\overline{f_{\e,\delta}}$ 
is explicitly constructed by using a partition of unity 
(see Proposition \ref{DW-iso}). 
It is easy to see that 
$g_{\e, \delta}(\mu \gamma_{\e, \delta})=\dbar v_{\e, \delta}$ and 
$g_{\e, \delta}(\alpha_{\e, \delta})=U_{\e, \delta}+\dbar \widetilde{v}_{\e, \delta}$ 
hold for some $v_{\e, \delta}$ and $\widetilde{v}_{\e, \delta}$ 
by the De Rham-Weil isomorphism. 
In particular, we have $U_{\e, \delta}
=\dbar(v_{\e, \delta} - \widetilde{v}_{\e, \delta})$
by $\mu \gamma_{\e, \delta}=\alpha_{\e, \delta}$.  
The important point here is that 
we can explicitly compute $v_{\e, \delta}$ and $\widetilde{v}_{\e, \delta}$ 
by using the partition of unity, $\beta^{\e,\delta}_{i_{0}...i_{k}}$, 
and $\gamma_{\e,\delta}$. 
From this explicit expression, we obtain the $L^2$-estimate for 
$v_{\e, \delta}$ and $\widetilde{v}_{\e, \delta}$.
(In the case $q=1$, we have already obtained the $L^2$-estimate.)
See \cite[Claim 5.11 and 5.13]{Mat13} for the precise argument. 
\end{proof}

We close this subsection with the following corollary\,:

\begin{cor}\label{c-sol}
For every $c$ with $c < \sup_{X}\Phi$, there exist  
$v_{\e, \delta} \in L^{n,q-1}_{(2)}(F^{m+1})_{\e, \delta}$ 
with the following property\,$:$ 
\begin{itemize}
\item[$\bullet$] $\dbar v_{\e, \delta}=su_{\e, \delta}$. 
\item[$\bullet$] $\varlimsup_{\delta \to 0} \|v_{\e,\delta}\|_{X_{c}, \e, \delta} $  can be bounded by a constant independent of $\e$.  
\end{itemize}
\end{cor}
\begin{proof}
Take $w_{\e,\delta}$ with the properties in Proposition \ref{sol}. 
On the other hand, we are assuming that the cohomology class $sA=\{su\}$ is zero, and thus there exists $w$ such that $\dbar w = su$ and $\| w \|_{X_{c}, Hh^{m}, \omega}<\infty$. 
Then $F$-valued $(n,q-1)$-form 
$v_{\e, \delta}$ defined by $v_{\e, \delta}:=w -sw_{\e,\delta}$ 
satisfies the desired properties  by $\sup_{X}|s|_{h_{\e}^{m}} \leq \sup_{X}|s|_{h^{m}}< \infty$. 
\end{proof}
\end{step}

\begin{step}[Asymptotics of norms of differential forms]\label{S5}
In this step, for every $b$ with $b< \sup_{X} \Phi$, 
we show that  
$$
\varliminf_{ \e \to 0}\varliminf_{ \delta \to 0}\|su_{\e, \delta}\|_{X_{b}, \e, \delta}=0. 
$$
This completes the proof by Proposition \ref{finish2}. 
For every $b$ with $b< \sup_{X} \Phi$, 
there exists $c$ such that $b<c< \sup_{X} \Phi$ and $d\Phi \not=0$ on $\partial X_{c}$ 
since the set of the critical values of $\Phi$ has Lebesgue measure zero by Sard's theorem. 
Fix such $c$ in this step. 
We want to apply Proposition \ref{key} to 
$su_{\e, \delta}$ and $v_{\e, \delta}$, 
but we do not know whether $v_{\e, \delta}$ is smooth on $Y_{\e}$. 
For this reason, for given $\e, \delta>0$, 
we take smooth $F$-valued $(n,q-1)$-forms 
$\{v_{\e,\delta, k}\}_{k=1}^{\infty}$ such that 
$v_{\e,\delta, k}$ (resp. $\dbar v_{\e,\delta, k}$) converges to 
$v_{\e, \delta}$ (resp. $\dbar v_{\e, \delta} = su_{\e, \delta}$) in the $L^2$-space  
$L^{n,\bullet}_{(2)}(F^{m+1})_{\e, \delta}$ (see Lemma \ref{density}). 
From now on, we consider only $d(>c)$
satisfying the properties in Proposition \ref{key} 
for countably many differential forms (see Remark \ref{key-rem}). 
Then Proposition \ref{key} yields  
\begin{align*}
\varliminf_{\e \to 0}\varliminf_{\delta \to 0}
\|su_{\e, \delta}\|^2_{X_{b}, \e, \delta}
&\leq \varliminf_{\e \to 0}
\varliminf_{\delta \to 0}
\|su_{\e, \delta}\|^2_{X_{d}, \e, \delta}\\
&= \varliminf_{\e \to 0} \varliminf_{\delta \to 0}
\lim_{k\to \infty}
\lla su_{\e, \delta}, \dbar v_{\e,\delta, k} \rra_{X_{d}, \e, \delta}\\
&=\varliminf_{\e \to 0}\varliminf_{\delta \to 0} 
\big\{ \lim_{k\to \infty} \lla \dbar^{*}_{\e, \delta} su_{\e, \delta},  v_{\e,\delta, k} \rra_{X_{d}, \e, \delta}
+\lim_{k\to \infty} \la (\dbar \Phi)^{*} su_{\e, \delta}, v_{\e,\delta, k} \ra_{\partial X_{d}, \e, \delta} \big\}. 
\end{align*}
Note that $(\dbar \Phi)^{*}$ is the adjoint operator of 
the wedge product $\dbar \Phi \wedge \bullet$
with respect to $\omega_{\e, \delta}$. 
We will show that the first term (resp. the second term) is zero 
in Proposition \ref{1st} (resp. in Proposition \ref{2nd}). 
For this purpose, we first prove the following proposition. 

\begin{prop}\label{har}
Under the above situation, we have 
\begin{align*}
\lim_{\e \to 0} \varlimsup_{\delta \to 0}\|D'^{*}_{\e,\delta}u_{\e, \delta}\|_{\e,\delta}=0. 
\end{align*}
Moreover we have 
\begin{align*}
\lim_{\e \to 0} \varlimsup_{\delta \to 0}\|
D_{\e, \delta}'^{*} su_{\e, \delta} \|_{\e, \delta}=0 \quad \text{ and } \quad
\lim_{\e \to 0} \varlimsup_{\delta \to 0}
\|\dbar^{*}_{\e, \delta} su_{\e, \delta} \|_{\e, \delta}=0. 
\end{align*}
\end{prop}
\begin{proof}
By applying the Bochner-Kodaira-Nakano identity 
(Proposition \ref{Nak} of the case $\Phi \equiv 0$) to $u_{\e, \delta}$ 
and $su_{\e, \delta}$, 
we obtain 
\begin{align}
& 
0=\|D_{\e, \delta}'^{*}u_{\e, \delta}\|^{2}_{\e, \delta}+\int_{Y_{\e}} g_{\e, \delta}\, dV_{\e, \delta}, \label{BKN-1} \\
&
\|\dbar^{*}_{\e, \delta} su_{\e, \delta} \|^2_{\e, \delta}= 
\|D_{\e, \delta}'^{*}su_{\e, \delta}\|^{2}_{\e, \delta}+\int_{Y_{\e}} |s|^2_{h^m_{\e}} g_{\e, \delta}\, dV_{\e, \delta}. \label{BKN-2}
\end{align}
Here we used the equality $\dbar su_{\e, \delta}=s\dbar u_{\e, \delta}=0$ 
and the fact that $u_{\e, \delta}$ is harmonic 
with respect to $H_{\e}$, $\omega_{\e, \delta}$. 
The integrand $g_{\e, \delta}$ is the function defined by 
$g_{\e}:=\langle \sqrt{-1}\Theta_{H_{\e}}\Lambda_{\e, \delta} u_{\e, \delta}, u_{\e, \delta}\rangle_{\e, \delta}$. 
By property (d) and property (B), 
we have 
\begin{align*}
\sqrt{-1}\Theta_{H_{\e}}(F) 
=\sqrt{-1}\Theta_{h_{\e}}(F) + \deldel \chi (\Phi)
\geq -\e \omega 
\geq -\e \omega_{\e, \delta}. 
\end{align*}
From the above inequalities, 
we can obtain 
\begin{align} \label{cur}
g_{\e} \geq -\e q |u_{\e, \delta}|^2_{\e, \delta}.  
\end{align}
(For example, see \cite[Step 2]{Mat13}).
Therefore we obtain 
$$
0 \geq 
\int_{\{g_{\e, \delta} \leq 0\}} g_{\e, \delta}\, dV_{\e, \delta}
\geq -\e q \int_{\{g_{\e, \delta} \leq 0\}} |u_{\e, \delta}|^2_{\e, \delta}\, dV_{\e, \delta}
\geq-\e q \|u_{\e, \delta}\|^2_{\e, \delta} 
\geq -\e q \|u\|^2_{H, \omega}
$$
from inequality $(\ref{ineq-2})$. 
By equality (\ref{BKN-1}), 
we obtain 
\begin{align*}
& 
\|D_{\e, \delta}'^{*}u_{\e, \delta}\|^{2}_{\e, \delta}+
\int_{\{g_{\e, \delta} \geq 0\}}g_{\e, \delta}\, dV_{\e, \delta} \leq 
-\int_{\{g_{\e, \delta} \leq 0\}}g_{\e, \delta}\, dV_{\e, \delta} \leq \e q \|u\|^2_{H, \omega}. 
\end{align*}
Hence we can see that
$$
\lim_{\e \to 0} \varlimsup_{\delta \to 0}\int_{\{g_{\e} \geq 0\}} g_{\e}\, dV_{\e, \delta}=0 
\quad \text{ and } \quad
\lim_{\e \to 0} \varlimsup_{\delta \to 0}\|D'^{*}_{\e,\delta}u_{\e, \delta}\|_{\e,\delta}=0. 
$$ 
On the other hand, by $\sup_{X}|s|_{h_{\e}^{m}}\leq \sup_{X}|s|_{h^{m}}<\infty$, 
we have 
\begin{align*}
& 
\int_{Y_{\e}} |s|^2_{h^m_{\e}} g_{\e}\, dV_{\e,\delta}
\leq \int_{\{g_{\e} \geq 0\}} |s|^2_{h^m_{\e}} g_{\e}\, dV_{\e,\delta}
\leq \sup_{X} |s|^2_{h^m} \int_{\{g_{\e} \geq 0\}}  g_{\e}\, dV_{\e,\delta}, \\ 
& 
\|D_{\e, \delta}'^{*}su_{\e, \delta}\|_{\e,\delta}
=\|-*\dbar*su_{\e, \delta}\|_{\e,\delta}
=\|sD_{\e, \delta}'^{*}u_{\e, \delta}\|_{\e,\delta}
\leq \sup_{X}|s|_{h^{m}} \|D_{\e, \delta}'^{*}u_{\e, \delta}\|_{\e,\delta}. 
\end{align*}
These inequalities and equality (\ref{BKN-2}) lead to the conclusion. 
\end{proof}

\begin{prop}\label{1st}
Under the above situation, we have 
$$
\lim_{\e \to 0} \varlimsup_{\delta \to 0}\lim_{k \to \infty}
\lla \dbar^{*}_{\e, \delta} su_{\e, \delta},  v_{\e,\delta, k} \rra_{X_{d}, \e, \delta}=0.
$$ 

\end{prop}
\begin{proof}
Cauchy-Schwarz's inequality yields  
$$
\lla \dbar^{*}_{\e, \delta} su_{\e, \delta},  
v_{\e,\delta, k} \rra_{X_{d}, \e, \delta}
\leq 
\| \dbar^{*}_{\e, \delta} su_{\e, \delta}\|_{X_{d}, \e, \delta}
\| v_{\e,\delta, k} \|_{X_{d}, \e, \delta}. 
$$
By the construction of $v_{\e,\delta, k}$, we may assume that  
$$
\varlimsup_{\e \to 0} \varlimsup_{\delta \to 0} \lim_{k \to \infty}
\| v_{\e,\delta, k} \|_{X_{d}, \e, \delta}
=\varlimsup_{\e \to 0} \varlimsup_{\delta \to 0} 
\| v_{\e, \delta} \|_{X_{d}, \e, \delta} 
$$
is finite (see Corollary \ref{c-sol}). 
On the other hand, 
the $L^2$-norm $\| \dbar^{*}_{\e, \delta} su_{\e, \delta}\|_{\e, \delta}$ converges to 
zero by Proposition \ref{har}. 
\end{proof}

We prove the following proposition by using the twisted 
Bochner-Kodaira-Nakano identity, which completes the proof of Theorem \ref{main}. 
\begin{prop}\label{2nd}
Under the above situation, 
we have 
$$
\varliminf_{\e \to 0} \varliminf_{\delta \to 0} \lim_{k \to \infty}
\la (\dbar \Phi)^{*} su_{\e, \delta}, v_{\e,\delta, k} 
\ra_{\partial X_{d}, \e, \delta}=0 
$$ 
for almost all $d$. 
\end{prop}

\begin{proof}
Cauchy-Schwarz's inequality and H\"older's inequality yield   
\begin{align*}
\la (\dbar \Phi)^{*} su_{\e, \delta}, v_{\e,\delta, k} 
\ra_{\partial X_{d}, \e, \delta}
&= \int_{\partial X_{d}} \big\langle (\dbar \Phi)^{*} su_{\e, \delta},
 v_{\e,\delta, k} 
\big\rangle_{\e, \delta}\, dS_{\omega_{\e, \delta}}\\
&\leq  \int_{\partial X_{d}} \big| (\dbar \Phi)^{*} su_{\e, \delta}\big|_{\e, \delta}
\big|v_{\e,\delta, k} \big|_{\e, \delta}\, dS_{\omega_{\e, \delta}}\\
&\leq  
\la (\dbar \Phi)^{*} su_{\e, \delta}, (\dbar \Phi)^{*} su_{\e, \delta}\ra_{\partial X_{d}, \e, \delta}\ 
\la v_{\e,\delta, k}, v_{\e,\delta, k} \ra_{\partial X_{d}, \e, \delta}. 
\end{align*}
We first show that the limit of 
$\la v_{\e,\delta, k}, v_{\e,\delta, k} \ra_{\partial X_{d}, \e, \delta}$ 
is finite for almost all $d$. 
By Fubini's theorem and $dV_{\e,\delta}=d \Phi \wedge dS_{\omega_{\e,\delta}}$, we have 
\begin{align*}
\int_{d \in (c-a, c+a)}
\la v_{\e,\delta, k}, v_{\e,\delta, k} \ra_{\partial X_{d}, \e, \delta}\, d\Phi
=\int_{\{c-a <\Phi< c+a\}} |v_{\e,\delta, k}|^{2}_{\e, \delta}\, dV_{\e, \delta}. 
\end{align*}
Further,  by Fatou's lemma, we have 
\begin{align*}
\int_{d \in (c-a, c+a)}
\varliminf_{\e \to 0} \varliminf_{\delta \to 0} \varliminf_{k \to \infty}
\la v_{\e,\delta, k}, v_{\e,\delta, k} \ra_{\partial X_{d}, \e, \delta}\, d\Phi
\leq \varliminf_{\e \to 0} \varliminf_{\delta \to 0} 
\|v_{\e, \delta}\|^{2}_{X_{c+a},\e, \delta}. 
\end{align*}
We are assuming that the right hand side is finite 
by Corollary \ref{c-sol}. 
Therefore the integrand of the left hand side must be finite for almost all 
$d \in (c-a, c+a)$.

Finally we will show that the norm of 
$(\dbar \Phi)^{*} su_{\e, \delta}$ on $\partial X_{d}$ converges to zero for almost all $d$. 
By $(\dbar \Phi)^{*} su_{\e, \delta}=s (\dbar \Phi)^{*} u_{\e, \delta}$ and 
$\sup|s|_{h_{\e}^{m}}\leq \sup|s|_{h^{m}}<\infty$, 
we have 
$$
\la (\dbar \Phi)^{*} su_{\e, \delta}, (\dbar \Phi)^{*} 
su_{\e, \delta}\ra_{\partial X_{d}, \e, \delta} 
\leq \sup_{X}|s|^2_{h^{m}} \la (\dbar \Phi)^{*} u_{\e, \delta}, (\dbar \Phi)^{*} 
u_{\e, \delta}\ra_{\partial X_{d}, \e, \delta}.  
$$
Hence it is sufficient to show 
the norm of $(\dbar \Phi)^{*} u_{\e, \delta}$ converges to zero. 
By applying Proposition \ref{key} to $(\dbar \Phi)^{*} u_{\e, \delta}$ and $u_{\e, \delta}$, 
we obtain 
\begin{align*}
\lla \dbar \big( (\dbar \Phi)^{*} u_{\e, \delta} \big), u_{\e, \delta} \rra_{X_{d}, \e, \delta}
&=\lla  (\dbar \Phi)^{*} u_{\e, \delta}, \dbar_{\e, \delta}^{*}   u_{\e, \delta} \rra_{X_{d}, \e, \delta}+
\la  (\dbar \Phi)^{*} u_{\e, \delta}, (\dbar \Phi)^{*} u_{\e, \delta} \ra_{\partial X_{d}, \e, \delta}\\
&=\la  (\dbar \Phi)^{*} u_{\e, \delta}, (\dbar \Phi)^{*} u_{\e, \delta} \ra_{\partial X_{d}, \e, \delta}. 
\end{align*}
Here we used the equality $\dbar^{*}_{\e, \delta} u_{\e, \delta}=0$. 
For the proof, we will compute the left hand side. 
Note that we have the equalities
 $\dbar u_{\e, \delta}=0$, $\partial \Phi \wedge u_{\e, \delta}=0$ 
and $\deldel\Phi \wedge u_{\e, \delta}=0$ since  
$u_{\e, \delta}$ is a $\dbar$-closed $F$-valued $(n,q)$-form. 
Therefore, by Lemma \ref{koukan}, we obtain  
\begin{align}
&\lla \dbar \big( (\dbar \Phi)^{*} u_{\e, \delta} \big), 
u_{\e, \delta} \rra_{X_{d}, \e, \delta}\notag\\
=&
-\lla  \partial \Phi \wedge (D_{\e, \delta}'^{*} u_{\e, \delta}), 
u_{\e, \delta} \rra_{X_{d}, \e, \delta}
+\lla  \sqrt{-1}\partial \dbar \Phi \Lambda u_{\e, \delta}, 
u_{\e, \delta} \rra_{X_{d}, \e, \delta}.  \label{aa}
\end{align}
From Lemma \ref{mul}, inequality $(\ref{ineq-2})$, and 
Cauchy-Schwartz's inequality, 
we can estimate the first term of equality (\ref{aa}) as follows\,: 
\begin{align*}
| \lla  \partial \Phi \wedge (D_{\e, \delta}'^{*} u_{\e, \delta}), 
u_{\e, \delta} \rra_{X_{d}, \e, \delta} | 
\leq \sup |\partial \Phi|_{\omega_{\e,\delta}} \|D_{\e, \delta}'^{*} u_{\e, \delta}\|_{\e,\delta} \|u_{\e, \delta}\|_{\e,\delta}
\leq \sup |\partial \Phi|_{\omega} \|D_{\e, \delta}'^{*} u_{\e, \delta}\|_{\e,\delta} \|u\|_{H, \omega}. 
\end{align*}
The right hand side converges to zero 
by Proposition \ref{har}. 

To estimate the second term of equality (\ref{aa}), 
by applying Ohsawa-Takegoshi's twisted Bochner-Kodaira-Nakano 
identity (Proposition \ref{Nak}), 
we obtain  
\begin{align*}
\|\sqrt{\eta}(\dbar \Phi) u_{\e, \delta}\|^2_{\e, \delta}&=
\|\sqrt{\eta}(D_{\e, \delta}'^{*}-(\partial \Phi)^{*}) u_{\e, \delta}\|^2_{\e, \delta}
+\lla \eta \sqrt{-1} (\Theta_{H_{\e}} +\partial \dbar \Phi)\Lambda u_{\e, \delta}, 
u_{\e, \delta}\rra_{\e, \delta}\\ 
&\geq \|\sqrt{\eta}(D_{\e, \delta}'^{*}-(\partial \Phi)^{*}) u_{\e, \delta}\|^2_{\e,\delta} 
-\e C \sup_{X}\eta\, \|u_{\e, \delta}\|^2_{\e, \delta}+
\lla  \sqrt{-1}\partial \dbar \Phi\Lambda u_{\e, \delta}, u_{\e, \delta}\rra_{\e, \delta},  
\end{align*} 
where $\eta$ is the bounded function defined by $\eta:=e^{\Phi}$. 
The above inequality follows from inequality $(\ref{cur})$. 
We compute the first term in the right hand side 
by using Lemma \ref{koukan} and Cauchy-Schwarz's inequality. 
It is easy to check that 
\begin{align*}
&\|\sqrt{\eta}(D_{\e, \delta}'^{*}u_{\e, \delta}-(\partial \Phi)^{*}) u_{\e, \delta}\|^2_{\e, \delta}\\
\geq & 
\|\sqrt{\eta}D_{\e, \delta}'^{*} u_{\e, \delta} \|^2_{\e, \delta}
-2 \|\sqrt{\eta}D_{\e, \delta}'^{*} u_{\e, \delta} \|_{\e, \delta}
\|\sqrt{\eta}(\partial \Phi)^{*} u_{\e, \delta} \|_{\e, \delta}
+\|\sqrt{\eta}(\partial \Phi)^{*} u_{\e, \delta} \|^2_{\e, \delta}
\\
\geq & 
-2 \|\sqrt{\eta}D_{\e, \delta}'^{*} u_{\e, \delta} \|_{\e, \delta}
\|\sqrt{\eta}(\partial \Phi)^{*} u_{\e, \delta} \|_{\e, \delta}
+\|\sqrt{\eta}(\partial \Phi)^{*} u_{\e, \delta} \|^2_{\e, \delta}. 
\end{align*}
On the other hand,   
Lemma \ref{koukan} implies 
$|(\partial \Phi)^{*}u_{\e, \delta}|^2=
|(\dbar \Phi) u_{\e, \delta}|^2+|(\dbar \Phi)^{*} u_{\e, \delta}|^2$, 
and thus we obtain 
\begin{align*}
\|\sqrt{\eta}(\partial \Phi)^{*} u_{\e, \delta} \|^2_{\e, \delta}
&=
\|\sqrt{\eta}(\dbar \Phi) u_{\e, \delta} \|^2_{\e, \delta}
+\|\sqrt{\eta}(\dbar \Phi)^{*} u_{\e, \delta} \|^2_{\e, \delta}\\
&\geq \|\sqrt{\eta}(\dbar \Phi) u_{\e, \delta} \|^2_{\e, \delta}
\end{align*}
From these inequalities, we have
\begin{align*}
\e C \sup_{X}\eta \|u_{\e, \delta}\|^2_{\e, \delta}
+
2 \|\sqrt{\eta}D_{\e,\delta}'^{*} u_{\e, \delta} \|^2_{\e, \delta}
\|\sqrt{\eta}(\partial \Phi)^{*} u_{\e, \delta} \|^2_{\e, \delta}
\geq 
\lla  \sqrt{-1}\partial \dbar \Phi\Lambda u_{\e, \delta}, u_{\e, \delta}\rra_{\e, \delta}\geq 0. 
\end{align*} 
The norm 
$\|\sqrt{\eta}D_{\e,\delta}'^{*} u_{\e, \delta} \|^2_{\e, \delta}$
converges to zero by Proposition \ref{har} 
and the norm $\|\sqrt{\eta}(\partial \Phi)^{*} u_{\e, \delta} \|^2_{\e, \delta}$ 
is uniformly bounded by Lemma \ref{mul}. 
This completes the proof. 
\end{proof}
\end{step}

\subsection{Proof of Theorem \ref{main2}}\label{Sec3-2}
In this subsection, 
we explain how to modify the proof of Theorem \ref{main} 
to obtain Theorem \ref{main2}.

\begin{theo}[Theorem \ref{main2}] \label{main2a} 
Let $\pi \colon X \to \Delta$ be a surjective 
proper K\"ahler morphism 
from a complex manifold $X$ to an analytic space $\Delta$. 
Let $(F,h)$ be a 
$($possibly$)$ singular hermitian line bundle on $X$ 
and $(M,h_{M})$ be a smooth hermitian line bundle on $X$. 
Assume that 
\begin{equation*}
\sqrt{-1}\Theta_{h_M}(M)\geq 0 \quad \text{and}\quad 
\sqrt{-1}(\Theta_h(F)-b \Theta 
_{h_M}(M))\geq 0
\end{equation*}
for some $b>0$. 
Then, for a non-zero $($holomorphic$)$ section $s$ of $M$, 
the multiplication map induced by the tensor product with $s$ 
$$
R^q \pi_{*}(K_{X}\otimes F \otimes \I{h}) 
\xrightarrow{\otimes s} 
R^q \pi_{*}(K_{X}\otimes F \otimes \I{h} \otimes M) 
$$
is injective for every $q$. 
\end{theo}

\begin{proof}
The proof is a slight revision of Theorem \ref{main}. 
We give only several differences with the proof of Theorem \ref{main}.

In Step \ref{S1}, 
by applying Theorem \ref{equi} for $\gamma =b \sqrt{-1}\Theta _{h_M}(M)$,  
we take a family of singular metrics 
$\{h_{\e} \}_{1\gg \e>0}$ on $F$ with 
the following properties\,: 
\begin{itemize}
\item[(a)] $h_{\e}$ is smooth on $X \setminus Z_{\e}$ 
for some proper subvariety $Z_{\e}$. 
\item[(b)]$h_{\e''} \leq h_{\e'} \leq h$ holds on $X$ 
for any $0< \e' < \e'' $.
\item[(c)]$\I{h}= \I{h_{\e}}$ on $X$.
\item[(e)]$\sqrt{-1} \Theta_{h_{\e}}(F) \geq 
b\sqrt{-1}\Theta_{h_{M}}(M)-\e \omega$ on $X$. 
\end{itemize}
Note that property (e) is obtained from the assumption 
$\sqrt{-1}\Theta_h(F) \geq b \sqrt{-1} \Theta _{h_M}(M)$. 
We can see that property (e) is stronger than 
property (d) in the proof of Theorem \ref{main}. 
Indeed, by the assumption $\sqrt{-1} \Theta _{h_M}(M) \geq 0$, 
we obtain property (d)  
\begin{itemize}
\item[(d)]$\sqrt{-1} \Theta_{h_{\e}}(F) \geq -\e \omega$ on $X$. 
\end{itemize}
By property (d), we can see that the same argument as in Step \ref{S2} works. 

In Step \ref{S3}, 
by considering the norm 
$\|su_{\e, \delta}\|_{\e, \delta}:= 
\|su_{\e, \delta}\|_{ H_{\e} h_{M}, \omega_{\e,\delta}}$ 
instead of 
$\|su_{\e, \delta}\|_{H_{\e}{h_{\e}}^{m}, \omega_{\e,\delta}}$, 
we can easily prove the same conclusion as in Proposition \ref{finish2}. 

In Step \ref{S4}, 
we can obtain $v_{\e} \in L^{n,q-1}_{(2)}(F\otimes M)_{\e, \delta}$ 
with the properties Corollary \ref{c-sol},  
since we do not use the line bundle $M$ when we prove Proposition \ref{sol}.

In Step \ref{S5}, we need to prove the following proposition (see Proposition \ref{har}). 
Recall that Proposition \ref{1st} and Proposition \ref{2nd} finish the proof of Theorem \ref{main} and they are obtained from Proposition \ref{har}.

\begin{prop}[cf. Proposition \ref{har}]\label{har2}
We have 
\begin{align*}
\lim_{\e \to 0} \varlimsup_{\delta \to 0}\|D'^{*}_{\e,\delta}u_{\e, \delta}\|_{\e,\delta}=0. 
\end{align*}
Moreover we have 
\begin{align*}
\lim_{\e \to 0}\varlimsup_{\delta \to 0}\|
D_{\e, \delta}'^{*} su_{\e, \delta} \|_{\e, \delta}=0 \quad \text{ and } \quad
\lim_{\e \to 0}\varlimsup_{\delta \to 0}
\|\dbar^{*}_{\e, \delta} su_{\e, \delta} \|_{\e, \delta}=0. 
\end{align*}
\end{prop}
\begin{proof}[Proof of Proposition \ref{har2}]
By applying the Bochner-Kodaira-Nakano identity to $u_{\e, \delta}$ 
and $su_{\e, \delta}$, 
we obtain the following equalities\,$:$
\begin{align*}
& 0=\|D_{\e, \delta}'^{*}u_{\e, \delta}\|^{2}_{\e, \delta}+
\int_{Y_{\e}} g_{\e, \delta}\, dV_{\e, \delta}.  \\
&
\|\dbar^{*}_{\e, \delta} su_{\e, \delta} \|^2_{\e, \delta}= 
\|D_{\e, \delta}'^{*}su_{\e, \delta}\|^{2}_{\e, \delta}+\int_{Y_{\e}} |s|^2_{h_{M}} (f_{\e, \delta} + g_{\e, \delta})\, dV_{\e, \delta}. 
\end{align*}
where the integrands $g_{\e, \delta}$ and $f_{\e,\delta}$
are the functions defined by 
\begin{align*}
g_{\e,\delta}:=\langle \sqrt{-1}\Theta_{H_{\e}}(F) \Lambda_{\e, \delta} u_{\e, \delta}, u_{\e, \delta}\rangle_{\e, \delta}, \\
f_{\e,\delta}:=\langle \sqrt{-1}\Theta_{h_{M}}(M) \Lambda_{\e, \delta} u_{\e, \delta}, u_{\e, \delta}\rangle_{\e, \delta}. 
\end{align*}
Since we have property (d), 
we obtain 
\begin{align} 
g_{\e} \geq -\e q |u_{\e, \delta}|^2_{\e, \delta}. 
\end{align}
By the same argument as in Proposition \ref{har}, 
we can see that 
\begin{align*}
\lim_{\e \to 0}\varlimsup_{\delta \to 0} \|D_{\e, \delta}'^{*}u_{\e, \delta}\|^{2}_{\e, \delta}=0 \quad \text{ and } \quad 
\lim_{\e \to 0}\varlimsup_{\delta \to 0}
\int_{\{g_{\e, \delta} \geq 0\}}g_{\e, \delta}\, dV_{\e, \delta}=0. 
\end{align*}
Therefore we can see that 
$\|D_{\e, \delta}'^{*}su_{\e, \delta}\|_{\e,\delta}=
\|sD_{\e, \delta}'^{*}u_{\e, \delta}\|_{\e,\delta}$ converges to zero 
from $\sup_{X}|s|_{h_{M}}<\infty$. 
On the other hand, 
from property (e), 
we can easily check   
$$
f_{\e,\delta} \leq 
\frac{1}{b}(g_{\e,\delta} + \e q |u_{\e,\delta}|_{\e,\delta}^{2}). 
$$
This implies that 
\begin{align*}
\int_{Y_{\e}} |s|^2_{h_{M}} (f_{\e, \delta} + g_{\e, \delta})\, dV_{\e, \delta}
&\leq    \int_{Y_{\e}}
|s|^2_{h_{M}}\bigg\{ (1+\frac{1}{b})g_{\e, \delta} + 
\frac{\e q}{b}  |u_{\e,\delta}|_{\e,\delta}^{2} \bigg\}\, dV_{\e, \delta}\\
&\leq    \int_{\{g_{\e, \delta} \geq 0\}}
|s|^2_{h_{M}}\bigg\{ (1+\frac{1}{b})g_{\e, \delta} + 
\frac{\e q}{b}  |u_{\e,\delta}|_{\e,\delta}^{2} \bigg\}\, dV_{\e, \delta} \\
&\leq  \sup_{X}|s|^2_{h_{M}} (1+\frac{1}{b}) \int_{\{g_{\e, \delta} \geq 0\}}
 g_{\e, \delta} \, dV_{\e, \delta} + 
\frac{\e q}{b} \|u\|_{H, \omega}^{2}. 
\end{align*}
This completes the proof. 
\end{proof}
By this proposition, 
we can prove the same conclusion as in Proposition \ref{1st} and Proposition \ref{2nd}. 
Therefore we obtain the conclusion. 
\end{proof}

\section{Applications}\label{Sec4}
\subsection{Proof of Corollary \ref{cor-1}}

In this subsection, we prove Corollary \ref{cor-1}.  

\begin{cor}[{Corollary \ref{cor-1}}]\label{cor-1a} 
Let $\pi : X \to \Delta$ be a surjective 
proper  K\"ahler morphism 
from a complex manifold $X$ to an analytic space $\Delta$, 
and $(F,h)$ be a $($possibly$)$ singular hermitian line bundle on $X$ 
with semi-positive curvature. 
Then the higher direct image sheaf 
$R^q \pi_{*}(K_{X}\otimes F \otimes \I{h})$ is torsion free 
for every $q$. 
Moreover, we obtain 
$$
R^q \pi_{*}(K_{X}\otimes F \otimes \I{h})=0
\quad \text{for every $q>\dim X - \dim \Delta$. }
$$
\end{cor} 

\begin{proof}
We apply Theorem \ref{main} 
in the case of $m=0$ 
to a holomorphic function $s$. 
For an open set $B \subset \Delta$ 
and a holomorphic function $s$ on $\pi^{-1}(B)$, 
the multiplication map 
$$
\Phi_{s}\colon R^q \pi_{*}(K_{X}\otimes F \otimes \I{h}) 
\xrightarrow{\otimes s} 
R^q \pi_{*}(K_{X}\otimes F \otimes \I{h})  
$$
is injective for every $q$. 
This implies that $R^q \pi_{*}(K_{X}\otimes F \otimes \I{h}) $ is torsion free. 
\end{proof}

\subsection{Proof of Theorem \ref{KVN}}
In this subsection, we prove Theorem \ref{KVN}. 
We first recall the definition of 
the numerical Kodaira dimension of singular hermitian line bundles on 
projective varieties (see \cite{Cao14} for K\"ahler manifolds). 

\begin{defi}[Numerical Kodaira dimension, \cite{Cao14}]\label{nd}
Let $(F,h)$ be a singular hermitian line bundle 
on a smooth projective variety $X$ such that $\sqrt{-1}\Theta_{h}(F) \geq 0$. 
Then the {\textit{numerical Kodaira dimension}} $\nd{F,h}$ of $(F,h)$ 
is defined to be $\nd{F,h}:=-\infty$ if $h\equiv \infty$, 
otherwise  
$$
\nd{F,h}:=\dim X - 
\varliminf_{\e  \to 0} \frac{\log {\rm{vol}}_{X}
{(A^{\e } \otimes F, h)}}{\log \e}
$$
where ${\rm{vol}}_{X}
{( A^{\e} \otimes F, h)}$ 
is defined by 
$$
{\rm{vol}}_{X}{( A^{\e } \otimes F, h)}:= 
\varlimsup_{m \to \infty} \frac{h^{0}(X,  A^{m\e} \otimes F^{m}\otimes \I{h^{m}}) 
}{m^{\dim X}}. 
$$
\end{defi}

By combining Cao's result in \cite{Cao14} 
with the openness theorem proved by Guan-Zhou in \cite{GZ13}, 
we have the following vanishing theorem. 
(See \cite{Hie14} and \cite{Lem14} for another proof for  the openness theorem.) 

\begin{theo}[{\cite[Theorem 1.3]{Cao14}, \cite[Theorem 1.1]{GZ13}}]\label{Cao-GZ}
Let $(F,h)$ be a singular hermitian line bundle 
on a compact K\"ahler manifold $X$ such that $\sqrt{-1}\Theta_{h}(F) \geq 0$. 
Then we have 
$$
H^{q}(X, K_{X}\otimes F \otimes \I{h})=0
\quad \text{for every $q>\dim X - \nd{F,h}$. }
$$
\end{theo}

We first prove Proposition \ref{lsc}. 
\begin{prop}[Proposition \ref{lsc}]\label{lsca}
Let $\pi \colon X \to \Delta$ be 
a surjective projective morphism 
from a complex manifold $X$ to an analytic space $\Delta$, 
and $(F,h)$ be a $($possibly$)$ singular hermitian line bundle on $X$ 
with semi-positive curvature. 
Assume that $\pi$ is smooth at a point $t_{0} \in \Delta$. 
Then, there exists a dense subset $Q \subset B$ 
in some open neighborhood $B$ of $t_{0}$ 
with the following properties\,$:$ 
$$\text{
For every $t \in Q $, 
we have $\nd{F|_{X_{t}},h|_{X_{t}}} \geq \nd{F|_{X_{t_{0}}},h|_{X_{t_{0}}}}$. 
}
$$
\end{prop}
\begin{rem}\label{lsc-rem}
By the proof of Proposition \ref{lsc}, 
we can add the property that 
$\I{h|_{X_{t}}^{m} } = \I{h^{m}} |_{X_{t}}$ for every $t \in Q$. 
\end{rem}

\begin{proof}
For a positive integer $m$, we define $Q_{m}$ by 
$$
Q_{m}:=\{t \in \Delta \,|\, \I{h|_{X_{t}}^{m} }= \I{h^{m}} |_{X_{t}}\}. 
$$
Note that we have $\I{h|_{X_{t}}^{m} } \subset \I{h^{m}} |_{X_{t}}$
by the Ohsawa-Takegoshi $L^2$-extension theorem. 
By Fubini's theorem, 
we can see that $\Delta \setminus Q_{m}$ has zero Lebesgue measure. 
We put $Q:=\cap_{m=1}^{\infty}Q_{m}$. 
Then $\Delta \setminus Q$ also has zero Lebesgue measure.
Let $A$ be a relatively ample line bundle $A$ on $X$. 
By the definition of the numerical dimension, 
it is sufficient to show that 
$$
h^{0}(X_{t}, \mathcal{O}_{X_{t}}(A^{m\e} \otimes F^{m})\otimes \I{h|_{X_{t}}^{m}})
\geq 
h^{0}(X_{t_{0}}, \mathcal{O}_{X_{t_{0}}}(A^{m\e } \otimes F^{m})\otimes \I{h|_{X_{t_{0}}}^{m}})
$$
for every $t \in Q$ and $m\gg 0$.

For the canonical bundle $K_{X}$ on $X$, 
we have 
$$
A^{m\e } \otimes F^{m}=K_{X} \otimes (A^{m\e } \otimes K_{X}^{-1})\otimes F^{m}. 
$$
$A^{m\e } \otimes K_{X}^{-1}$ admits 
a smooth (hermitian) metric $g_{m}$ with positive curvature 
for a sufficiently large $m \gg 0$. 
We can extend a basis $\{s_{i}\}_{i \in I}$ in 
$H^{0}(X_{t_{0}}, \mathcal{O}_{X_{t_{0}}}(A^{m\e } \otimes F^{m})
\otimes \I{h|_{X_{t_{0}}}^{m}}$ 
to sections $\{ \tilde{s}_{i} \}$ in 
$H^{0}(X, \mathcal{O}_{X}(A^{m\e } \otimes F^{m})\otimes \I{h^m})$, 
by applying the Ohsawa-Takegoshi $L^2$ extension theorem to 
$(A^{m\e } \otimes K_{X}^{-1}\otimes F^{m}, g_{m}h^{m})$ 
(see \cite{OT87} and \cite{Man93}).

We can easily see that $\{\tilde{s}_{i} |_{X_{t}}\}_{i \in I}$ 
is linearly independent in $H^{0}(X_{t}, \mathcal{O}_{X_{t}}(A^{m\e} \otimes F^{m}) 
\otimes \I{h^{m}}|_{X_{t}})$
for every $t$ in some open neighborhood $B$ of $t_{0}$. 
Indeed, if there exist a point $t$ converging to $t_{0}$ and 
$a_{t,i} \in \mathbb{C}$ 
such that $\sum_{i \in I} a_{t,i} \tilde{s}_{i} |_{X_{t}}=0$, 
then we may assume that $a_{t,i}$ converges to some $a_{i}$ as $t \to t_{0}$. 
As $t$ tends to $t_{0}$, 
we obtain $\sum_{i \in I} a_{i} \tilde{s}_{i} |_{X_{t_{0}}}=0$ 
from $\sum_{i \in I} a_{t,i} \tilde{s}_{i} |_{X_{t}}=0$. 
Therefore $\{\tilde{s}_{i} |_{X_{t}}\}_{i \in I}$ 
is linearly independent. 
If $t \in Q$, the restriction $\tilde{s}_{i} |_{X_{t}}$ to $X_{t}$ 
is a section in 
$H^{0}(X_{t}, \mathcal{O}_{X_{t}}(A^{m\e} \otimes F^{m}) 
\otimes \I{h|_{X_{t}}^{m}})$. 
This completes the proof. 
\end{proof}

\begin{theo}[Theorem \ref{KVN}] \label{KVNa}
Let $\pi \colon X \to \Delta$ be a surjective projective morphism 
from a complex manifold $X$ to an analytic space $\Delta$, 
and $(F,h)$ be a $($possibly$)$ singular hermitian line bundle on $X$ 
with semi-positive curvature. 
Then we have 
$$
R^q \pi_{*}(K_{X}\otimes F \otimes \I{h})=0
\quad \text{for every $q>f - \max_{\substack{\pi \text{ is smooth} \\ \text{ at }   t \in \Delta }}$} 
\nd{F|_{X_{t}},h|_{X_{t}}}, 
$$
where $f$ is the dimension of general fibers. 
In particular, 
if $(F|_{X_{t}},h|_{X_{t}})$ is big 
for some point $t$ in the smooth locus of $\pi$, 
then we have 
$$
R^q \pi_{*}(K_{X}\otimes F \otimes \I{h})=0
\quad \text{for every $q>0$}. 
$$
\end{theo}
\begin{proof}
We take a point $t_{0}$ with 
$$
\nd{F|_{X_{t_{0}}},h|_{X_{t_{0}}}}= 
\max_{\substack{\pi \text{ is smooth} \\ \text{ at }   t \in \Delta }}
\nd{F|_{X_{t}},h|_{X_{t}}}. 
$$
By Proposition \ref{lsc} and Remark \ref{lsc-rem}, 
we can take a dense subset $Q$ in some neighborhood $B$ of $t_{0}$ 
such that $\nd{F|_{X_{t}},h|_{X_{t}}} \geq \nd{F|_{X_{t_{0}}},h|_{X_{t_{0}}}}$ 
and $\I{h|_{X_{t}}^{m} } = \I{h^{m}} |_{X_{t}}$ 
for every $t \in Q$. 
Therefore, By Cao's result and the openness theorem (see Theorem \ref{Cao-GZ}), 
we obtain 
\begin{align*}
H^{q}(X_{t}, \mathcal{O}_{X_{t}}(K_{X} \otimes F)\otimes \I{h})&= 
H^{q}(X_{t}, \mathcal{O}_{X_{t}}(K_{X} \otimes F)\otimes \I{h|_{X_{t}}})
=0
\end{align*}
for $q>f-\nd{F|_{X_{t_{0}}},h|_{X_{t_{0}}}}  \geq f - \nd{F|_{X_{t}},h|_{X_{t}}}$ and for every $t \in Q \cap \Delta'$. 
Here $\Delta'$ is the Zariski open set in $\Delta$ defined by 
$$
\Delta':=\{t \in \Delta \, |\, \text{$\pi$ is smooth at $t$ and 
$R^{q}\pi_{*}(K_{X}\otimes F \otimes \I{h})$ is locally free at $t$.}\}
$$ 
By the flat base change theorem, 
we obtain $R^{q} \pi_{*}(K_{X} \otimes F \otimes \I{h})_{t}=0$ 
for every $t \in Q  \cap \Delta'$. 
This implies that 
$R^{q}\pi_{*}(K_{X}\otimes F \otimes \I{h})_{t}=0$ on $\Delta'$. 
We obtain the conclusion 
since $R^{q}\pi_{*}(K_{X}\otimes F \otimes \I{h})$ 
is torsion free. 

\end{proof}

\subsection{Proof of Corollary \ref{cor-2}}
Finally, we prove Corollary \ref{cor-2}.

\begin{cor}[Corollary \ref{cor-2}]\label{cor-2a}
Let $\pi \colon X \to \Delta$ be 
a surjective proper K\"ahler morphism 
from a complex manifold $X$ to 
an open disk $\Delta \subset \mathbb{C}$ and 
$(F,h)$ be a singular hermitian line bundle with semi-positive curvature. 
Then every section in 
$H^{0}(X_{0}, \mathcal{O}_{X_{0}}(K_{X}\otimes F))$ 
that comes from $H^{0}(X_{0}, \mathcal{O}_{X_{0}}(K_{X}\otimes F)\otimes \I{h})$ can be extended to a section in 
$H^{0}(X, \mathcal{O}_{X}(K_{X}\otimes F)\otimes \I{h})$ 
by replacing $\Delta$ with a smaller disk. 
In particular, 
if $K_{X}$ admits a singular metric $h$ whose curvature 
is semi-positive and Lelong number is zero at every point in $ X_{0}$, 
then Problem \ref{prob} is affirmatively solved. 
\end{cor}

\begin{proof}
Let $s$ be a holomorphic function on $X$ with $X_{0}=s^{-1}(0)$.
By Theorem \ref{main}, we can conclude  
$$
H^{1}(X, K_{X}\otimes F \otimes \I{h})
\xrightarrow{\otimes s} 
H^{1}(X, K_{X}\otimes F \otimes \I{h})
$$
is injective 
for a sufficiently small $\Delta$. 
On the other hand, since $X_{0}$ is a subvariety of codimension one and 
$R^q \pi_{*}(K_{X}\otimes F \otimes \I{h}) $ is 
torsion free, the following sequence is exact: 
\begin{align*}
0\to \mathcal{O}_{X}(K_{X}\otimes F \otimes \I{h}) \otimes 
\mathcal{I}_{X_{0}}
\to \mathcal{O}_{X}(K_{X}\otimes F \otimes \I{h})
\to \mathcal{O}_{X_{0}}(K_{X}\otimes F \otimes \I{h})
\to0. 
\end{align*}
The induced long exact sequence implies that 
for every section $t$ in 
$H^{0}(X_{0}, \mathcal{O}_{X_{0}}(K_{X}\otimes F \otimes \I{h}))$, 
there exists a section $T$ in 
$H^{0}(X, \mathcal{O}_{X}(K_{X}\otimes F \otimes \I{h}))$ 
such that $T|_{X_{0}}=t$. 
Further we have the following commutative diagram: 
$$
\begin{CD}
H^{0}(X, \mathcal{O}_{X}(K_{X}\otimes F \otimes \I{h})) 
@>>>
H^{0}(X_{0}, \mathcal{O}_{X_{0}}(K_{X}\otimes F \otimes \I{h}))  \\ 
@VVV  @VVV
\\ H^{0}(X, \mathcal{O}_{X}(K_{X}\otimes F )) @>>>
H^{0}(X_{0}, \mathcal{O}_{X_{0}}(K_{X}\otimes F )).  
\end{CD}
$$
Therefore we can extend a section in 
$H^{0}(X, \mathcal{O}_{X}(K_{X}\otimes F ))$
that comes from 
$H^{0}(X_{0}, \mathcal{O}_{X_{0}}(K_{X}\otimes F \otimes \I{h}))$ 
to $X$. 
\end{proof}


\end{document}